\documentclass{amsart}%
\usepackage{amssymb}
\usepackage{amsfonts}
\usepackage{amsmath}
\usepackage{graphicx}%
\providecommand{\U}[1]{\protect \rule{.1in}{.1in}}
\newtheorem{theorem}{Theorem}
\theoremstyle{plain}

\newtheorem{corollary}{Corollary}

\newtheorem{definition}{Definition}

\newtheorem{lemma}{Lemma}

\newtheorem{remark}{Remark}

\numberwithin{equation}{section}
\begin{document}
\title[multilinear $BMO$ estimates]{multilinear $BMO$ estimates for the commutators of multilinear fractional
maximal and integral{\ operators }on the product generalized Morrey spaces}
\author{F. GURBUZ}
\address{ANKARA UNIVERSITY, FACULTY OF SCIENCE, DEPARTMENT OF MATHEMATICS, TANDO\u{G}AN
06100, ANKARA, TURKEY }
\email{feritgurbuz84@hotmail.com}
\urladdr{}
\thanks{}
\thanks{}
\thanks{}
\date{}
\dedicatory{ }
\begin{abstract}
In this paper, we establish multilinear $BMO$ estimates for commutators of
multilinear fractional maximal and integral{\ operators both on product
generalized Morrey spaces and product generalized vanishing Morrey spaces},
respectively. Similar results are still valid for commutators of multilinear
maximal and singular integral{\ operators. }

\end{abstract}
\maketitle

\section{Introduction and main results}

Because of the need for study of the local behavior of solutions of second
order elliptic partial differential equations (PDEs) and together with the now
well-studied Sobolev Spaces, constitude a formidable three parameter family of
spaces useful for proving regularity results for solutions to various PDEs,
especially for non-linear elliptic systems, in 1938, Morrey \cite{Morrey}
introduced the classical Morrey spaces $L_{p,\lambda}$ which naturally are
generalizations of the classical Lebesgue spaces.

We will say that a function $f\in L_{p,\lambda}=L_{p,\lambda}\left(
{\mathbb{R}^{n}}\right)  $ if%
\begin{equation}
\sup_{x\in{\mathbb{R}^{n},r>0}}\left[  r^{-\lambda}%
%TCIMACRO{\dint \limits_{B(x,r)}}%
%BeginExpansion
{\displaystyle \int \limits_{B(x,r)}}
%EndExpansion
\left \vert f\left(  y\right)  \right \vert ^{p}dy\right]  ^{1/p}<\infty.
\label{1.1}%
\end{equation}
Here, $1<p<\infty$ and $0<\lambda<n$ and the quantity of (\ref{1.1}) is the
$\left(  p,\lambda \right)  $-Morrey norm, denoted by $\left \Vert f\right \Vert
_{L_{p,\lambda}}$. In recent years, more and more researches focus on function
spaces based on Morrey spaces to fill in some gaps in the theory of Morrey
type spaces (see, for example, \cite{GulJIA, Gurbuz1, Gurbuz2, Gurbuz3,
Gurbuz4, Gurbuz5, Karaman, Lin, Pal, Softova, Yu}). Moreover, these spaces are
useful in harmonic analysis and PDEs. But, this topic exceeds the scope of
this paper. Thus, we omit the details here. On the other hand, the study of
the operators of harmonic analysis in vanishing Morrey space, in fact has been
almost not touched. A version of the classical Morrey space $L_{p,\lambda
}({\mathbb{R}^{n}})$ where it is possible to approximate by "nice" functions
is the so called vanishing Morrey space $VL_{p,\lambda}({\mathbb{R}^{n}})$ has
been introduced by Vitanza in \cite{Vitanza1} and has been applied there to
obtain a regularity result for elliptic PDEs. This is a subspace of functions
in $L_{p,\lambda}({\mathbb{R}^{n}})$, which satisfies the condition%
\[
\lim_{r\rightarrow0}\sup_{\underset{0<t<r}{x\in{\mathbb{R}^{n}}}}\left[
t^{-\lambda}%
%TCIMACRO{\dint \limits_{B(x,t)}}%
%BeginExpansion
{\displaystyle \int \limits_{B(x,t)}}
%EndExpansion
\left \vert f\left(  y\right)  \right \vert ^{p}dy\right]  ^{1/p}=0,
\]
where $1<p<\infty$ and $0<\lambda<n$ for brevity, so that%
\[
VL_{p,\lambda}({\mathbb{R}^{n}})=\left \{  f\in L_{p,\lambda}({\mathbb{R}^{n}%
}):\lim_{r\rightarrow0}\sup_{\underset{0<t<r}{x\in{\mathbb{R}^{n}}}}%
t^{-\frac{\lambda}{p}}\Vert f\Vert_{L_{p}(B(x,t))}=0\right \}  .
\]
Later in \cite{Vitanza2} Vitanza has proved an existence theorem for a
Dirichlet problem, under weaker assumptions than in \cite{Miranda} and a
$W^{3,2}$ regularity result assuming that the partial derivatives of the
coefficients of the highest and lower order terms belong to vanishing Morrey
spaces depending on the dimension. For the properties and applications of
vanishing Morrey spaces, see also \cite{Cao-Chen}.

After studying Morrey spaces in detail, researchers have passed to the concept
of generalized Morrey spaces. Firstly, motivated by the work of \cite{Morrey},
Mizuhara \cite{Miz} introduced generalized Morrey spaces $M_{p,\varphi}$.
Then, Guliyev \cite{GulJIA} defined the generalized Morrey spaces
$M_{p,\varphi}$ with normalized norm as follows:

\begin{definition}
\cite{GulJIA} \textbf{(generalized Morrey space)\ }Let $\varphi(x,r)$ be a
positive measurable function on ${\mathbb{R}^{n}}\times(0,\infty)$. If
$0<p<\infty$, then the generalized Morrey space $M_{p,\varphi}\equiv
M_{p,\varphi}({\mathbb{R}^{n}})$ is defined by%
\[
\left \{  f\in L_{p}^{loc}({\mathbb{R}^{n}}):\Vert f\Vert_{M_{p,\varphi}}%
=\sup \limits_{x\in{\mathbb{R}^{n}},r>0}\varphi(x,r)^{-1}|B(x,r)|^{-\frac{1}%
{p}}\Vert f\Vert_{L_{p}(B(x,r))}<\infty \right \}  .
\]

\end{definition}

Obviously, the above definition recover the definition of $L_{p,\lambda
}({\mathbb{R}^{n}})$ if we choose $\varphi(x,r)=r^{\frac{\lambda-n}{p}}$, that
is
\[
L_{p,\lambda}\left(  {\mathbb{R}^{n}}\right)  =M_{p,\varphi}\left(
{\mathbb{R}^{n}}\right)  \mid_{\varphi(x,r)=r^{\frac{\lambda-n}{p}}}.
\]

Everywhere in the sequel we assume that $\inf \limits_{x\in{\mathbb{R}^{n}%
},r>0}\varphi(x,r)>0$ which makes the above spaces non-trivial, since the
spaces of bounded functions are contained in these spaces. We point out that
$\varphi(x,r)$ is a measurable nonnegative function and no monotonicity type
condition is imposed on these spaces.

In \cite{GulJIA}, \cite{Gurbuz3}, \cite{Karaman}, \cite{Lin}, \cite{Miz} and
\cite{Softova}, the boundedness of the maximal operator and
Calder\'{o}n-Zygmund operator on the generalized Morrey spaces $M_{p,\varphi}$
has been obtained, respectively.

For brevity, in the sequel we use the notations%
\[
\mathfrak{M}_{p,\varphi}\left(  f;x,r\right)  :=\frac{|B(x,r)|^{-\frac{1}{p}%
}\, \Vert f\Vert_{L_{p}(B(x,r))}}{\varphi(x,r)}%
\]
and%
\[
\mathfrak{M}_{p,\varphi}^{W}\left(  f;x,r\right)  :=\frac{|B(x,r)|^{-\frac
{1}{p}}\, \Vert f\Vert_{WL_{p}(B(x,r))}}{\varphi(x,r)}.
\]

In this paper, extending the definition of vanishing Morrey spaces
\cite{Vitanza1}, we introduce generalized vanishing Morrey spaces
$VM_{p,\varphi}({\mathbb{R}^{n}})$ with normalized norm in the following form.

\begin{definition}
\textbf{(generalized vanishing Morrey space) }The generalized vanishing Morrey
space $VM_{p,\varphi}({\mathbb{R}^{n}})$ is defined by%
\[
\left \{  f\in M_{p,\varphi}({\mathbb{R}^{n}}):\lim \limits_{r\rightarrow0}%
\sup \limits_{x\in{\mathbb{R}^{n}}}\mathfrak{M}_{p,\varphi}\left(
f;x,r\right)  =0\right \}  .
\]

\end{definition}

Everywhere in the sequel we assume that%
\begin{equation}
\lim_{r\rightarrow0}\frac{1}{\inf \limits_{x\in{\mathbb{R}^{n}}}\varphi
(x,r)}=0, \label{2}%
\end{equation}
and%
\begin{equation}
\sup_{0<r<\infty}\frac{1}{\inf \limits_{x\in{\mathbb{R}^{n}}}\varphi
(x,r)}<\infty, \label{3}%
\end{equation}
which make the spaces $VM_{p,\varphi}({\mathbb{R}^{n}})$ non-trivial, because
bounded functions with compact support belong to this space. The spaces
$VM_{p,\varphi}({\mathbb{R}^{n}})$ are Banach spaces with respect to the norm
\[
\Vert f\Vert_{VM_{p,\varphi}}\equiv \Vert f\Vert_{M_{p,\varphi}}=\sup
\limits_{x\in{\mathbb{R}^{n}},r>0}\mathfrak{M}_{p,\varphi}\left(
f;x,r\right)  .
\]
The spaces $VM_{p,\varphi}({\mathbb{R}^{n}})$ are also closed subspaces of the
Banach spaces $M_{p,\varphi}({\mathbb{R}^{n}})$, which may be shown by
standard means.

Furthermore, we have the following embeddings:%
\[
VM_{p,\varphi}\subset M_{p,\varphi},\qquad \Vert f\Vert_{M_{p,\varphi}}%
\leq \Vert f\Vert_{VM_{p,\varphi}}.
\]

On the other hand, it is well known that, for the purpose of researching
non-smoothness partial differential equation, mathematicians pay more
attention to the singular integrals. Moreover, the fractional type operators
and their weighted boundedness theory play important roles in harmonic
analysis and other fields, and the multilinear operators arise in numerous
situations involving product-like operations, see \cite{Chen1, Chen, Fu,
Grafakos-sm, Grafakos1, Grafakos2, Kenig, Moen, Si} for instance.

First of all, we recall some basic properties and notations used in this paper.

Let ${\mathbb{R}^{n}}$ be the $n$-dimensional Euclidean space of points
$x=(x_{1},...,x_{n})$ with norm $|x|=\left(
%TCIMACRO{\tsum \nolimits_{i=1}^{n}}%
%BeginExpansion
{\textstyle \sum \nolimits_{i=1}^{n}}
%EndExpansion
x_{i}^{2}\right)  ^{1/2}$ and corresponding $m$-fold product spaces $\left(
m\in%
%TCIMACRO{\U{2115} }%
%BeginExpansion
\mathbb{N}
%EndExpansion
\right)  $ be $\left(  {\mathbb{R}^{n}}\right)  ^{m}={\mathbb{R}^{n}%
\times \cdots \times \mathbb{R}^{n}}$. Let $B=B(x,r)$ denotes open ball centered
at $x$ of radius $r$ for $x\in{\mathbb{R}^{n}}$ and $r>0$ and $B^{c}(x,r)$ its
complement. Also $|B(x,r)|$ is the Lebesgue measure of the ball $B(x,r)$ and
$|B(x,r)|=v_{n}r^{n}$, where $v_{n}=|B(0,1)|$. We also denote by
$\overrightarrow{y}=\left(  y_{1},\ldots,y_{m}\right)  $, $d\overrightarrow
{y}=dy_{1}\ldots dy_{m}$, and by $\overrightarrow{f}$ the $m$-tuple $\left(
f_{1},...,f_{m}\right)  $, $m$, $n$ the nonnegative integers with $n\geq2$,
$m\geq1$.

Let $\overrightarrow{f}\in L_{p_{1}}^{loc}\left(  {\mathbb{R}^{n}}\right)
\times \cdots \times L_{p_{m}}^{loc}\left(  {\mathbb{R}^{n}}\right)  $. Then
multi-sublinear fractional maximal operator $M_{\alpha}^{\left(  m\right)  }$
is defined by%
\[
M_{\alpha}^{\left(  m\right)  }\left(  \overrightarrow{f}\right)  \left(
x\right)  =\sup_{t>0}\left \vert B\left(  x,t\right)  \right \vert
^{\frac{\alpha}{n}}\left[
%TCIMACRO{\dprod \limits_{i=1}^{m}}%
%BeginExpansion
{\displaystyle \prod \limits_{i=1}^{m}}
%EndExpansion
\frac{1}{\left \vert B\left(  x,t\right)  \right \vert }\int \limits_{B\left(
x,t\right)  }\left \vert f_{i}\left(  y_{i}\right)  \right \vert \right]
d\overrightarrow{y},\text{ \  \ }0\leq \alpha<mn.
\]
From definition, if $\alpha=0$ then $M_{\alpha}^{\left(  m\right)  }$ is the
multi-sublinear maximal operator $M^{\left(  m\right)  }$ and also; in the
case of $m=1$, $M_{\alpha}^{\left(  m\right)  }$ is the classical fractional
maximal operator $M_{\alpha}$.

The theory of multilinear {Calder\'{o}n-Zygmund singular integral operators,
originated from the works of Coifman and Meyer's \cite{CM}, plays an important
role in harmonic analysis. Its study has been attracting a lot of attention in
the last few decades. A systematic analysis of many basic properties of such
multilinear singular integral operators can be found in the articles by
Coifman-Meyer \cite{CM}, Grafakos-Torres \cite{Grafakos1, Grafakos2,
Grafakos3*}, Chen et al. \cite{Chen1}, Fu et al. \cite{Fu}.}

Let $T^{\left(  m\right)  }$ $\left(  m\in%
%TCIMACRO{\U{2115} }%
%BeginExpansion
\mathbb{N}
%EndExpansion
\right)  $ be a multilinear operator initially defined on the $m$-fold product
of Schwartz spaces and taking values into the space of tempered distributions,%
\[
T^{\left(  m\right)  }:S\left(  {\mathbb{R}^{n}}\right)  \times \cdots \times
S\left(  {\mathbb{R}^{n}}\right)  \rightarrow S\left(  {\mathbb{R}^{n}%
}\right)  .
\]
Following \cite{Grafakos1}, recall that the $m$(multi)-linear
{Calder\'{o}n-Zygmund operator }$T^{\left(  m\right)  }$ $\left(  m\in%
%TCIMACRO{\U{2115} }%
%BeginExpansion
\mathbb{N}
%EndExpansion
\right)  ${\ }for test vector $\overrightarrow{f}=\left(  f_{1},\ldots
,f_{m}\right)  $ is defined by%
\[
T^{\left(  m\right)  }\left(  \overrightarrow{f}\right)  \left(  x\right)  =%
%TCIMACRO{\dint \limits_{\left(  {\mathbb{R}^{n}}\right)  ^{m}}}%
%BeginExpansion
{\displaystyle \int \limits_{\left(  {\mathbb{R}^{n}}\right)  ^{m}}}
%EndExpansion
K\left(  x,y_{1},\ldots,y_{m}\right)  \left \{
%TCIMACRO{\dprod \limits_{i=1}^{m}}%
%BeginExpansion
{\displaystyle \prod \limits_{i=1}^{m}}
%EndExpansion
f_{i}\left(  y_{i}\right)  \right \}  dy_{1}\cdots dy_{m},\text{ \  \  \ }x\notin%
%TCIMACRO{\dbigcap \limits_{i=1}^{m}}%
%BeginExpansion
{\displaystyle \bigcap \limits_{i=1}^{m}}
%EndExpansion
suppf_{i},
\]
where $K$ is an $m$-{Calder\'{o}n-Zygmund kernel which is a locally integrable
function defined }off the diagonal $y_{0}=y_{1}=\cdots=y_{m}$ on $\left(
{\mathbb{R}^{n}}\right)  ^{m+1}$ satisfying the {following size estimate:}%
\[
\left \vert K\left(  x,y_{1},\ldots,y_{m}\right)  \right \vert \leq \frac
{C}{\left \vert \left(  x-y_{1},\ldots,x-y_{m}\right)  \right \vert ^{mn}},
\]
for some $C>0$ and some smoothness estimates, see \cite{Grafakos1, Grafakos2,
Grafakos3*} for details.

The result of Grafakos and Torres \cite{Grafakos1, Grafakos3*} shows that the
multilinear {Calder\'{o}n-Zygmund operator is bounded on the product of
Lebesgue spaces.}

\begin{theorem}
\label{TeoA}\cite{Grafakos1, Grafakos3*} Let $T^{\left(  m\right)  }$ $\left(
m\in%
%TCIMACRO{\U{2115} }%
%BeginExpansion
\mathbb{N}
%EndExpansion
\right)  $ be an $m$-linear {Calder\'{o}n-Zygmund operator. Then, for any
numbers }$1\leq p_{1},\ldots,p_{m},p<\infty$ with $\frac{1}{p}=\frac{1}{p_{1}%
}+\cdots+\frac{1}{p_{m}}$, $T^{\left(  m\right)  }$ can be extended to a
bounded operator from $L_{p_{1}}\times \cdots \times L_{p_{m}}$ into $L_{p}$,
and bounded from $L_{1}\times \cdots \times L_{1}$ into $L_{\frac{1}{m},\infty}$.
\end{theorem}

On the other hand, the multilinear fractional type operators are natural
generalization of linear ones. Their earliest version was originated on the
work of Grafakos {\cite{Grafakos-sm} }in 1992, in which he studied the
multilinear maximal function and multilinear fractional integral defined by
\[
M_{\alpha}^{\left(  m\right)  }\left(  \overrightarrow{f}\right)  \left(
x\right)  =\sup_{t>0}\frac{1}{r^{n-\alpha}}\int \limits_{\left \vert
y\right \vert <r}\left \vert
%TCIMACRO{\dprod \limits_{i=1}^{m}}%
%BeginExpansion
{\displaystyle \prod \limits_{i=1}^{m}}
%EndExpansion
f_{i}\left(  x-\theta_{i}y\right)  \right \vert dy
\]
and
\[
I_{\alpha}^{\left(  m\right)  }\left(  \overrightarrow{f}\right)  \left(
x\right)  =\int \limits_{{\mathbb{R}^{n}}}\frac{1}{\left \vert y\right \vert
^{n-\alpha}}%
%TCIMACRO{\dprod \limits_{i=1}^{m}}%
%BeginExpansion
{\displaystyle \prod \limits_{i=1}^{m}}
%EndExpansion
f_{i}\left(  x-\theta_{i}y\right)  dy,
\]
where $\theta_{i}\left(  i=1,\ldots,m\right)  $ are fixed distinct are nonzero
real numbers and $0<\alpha<n$. We note that, if we simply take $m=1$ and
$\theta_{i}=1$, then $M_{\alpha}$ and $I_{\alpha}$ are just the operators
studied by Muckenhoupt and Wheeden in \cite{Muckenhoupt}. In this paper we
deal with another kind of multilinear operator which was defined by {Kenig and
Stein \cite{Kenig} }for $\overrightarrow{f}=\left(  f_{1},\ldots,f_{m}\right)
$, which is called multilinear fractional integral operator as follows%
\[
I_{\alpha}^{\left(  m\right)  }\left(  \overrightarrow{f}\right)  \left(
x\right)  =%
%TCIMACRO{\dint \limits_{\left(  {\mathbb{R}^{n}}\right)  ^{m}}}%
%BeginExpansion
{\displaystyle \int \limits_{\left(  {\mathbb{R}^{n}}\right)  ^{m}}}
%EndExpansion
\frac{1}{\left \vert \left(  x-y_{1},\ldots,x-y_{m}\right)  \right \vert
^{mn-\alpha}}\left \{
%TCIMACRO{\dprod \limits_{i=1}^{m}}%
%BeginExpansion
{\displaystyle \prod \limits_{i=1}^{m}}
%EndExpansion
f_{i}\left(  y_{i}\right)  \right \}  d\overrightarrow{y},
\]
whose kernel is%
\[
\left \vert K\left(  x,y_{1},\ldots,y_{m}\right)  \right \vert =\left \vert
\left(  x-y_{1},\ldots,x-y_{m}\right)  \right \vert ^{-mn+\alpha},\text{
\  \ }0<\alpha<mn,
\]
where $f_{1},\ldots,f_{m}:{\mathbb{R}^{n}\rightarrow \mathbb{R}}$ are
measurable and $\left \vert \left(  x-y_{1},\ldots,x-y_{m}\right)  \right \vert
=\sqrt{%
%TCIMACRO{\tsum \limits_{i=1}^{m}}%
%BeginExpansion
{\textstyle \sum \limits_{i=1}^{m}}
%EndExpansion
\left \vert x-y_{i}\right \vert ^{2}}$.

They {\cite{Kenig}} proved that $I_{\alpha}^{\left(  m\right)  }$ $\left(
m\in%
%TCIMACRO{\U{2115} }%
%BeginExpansion
\mathbb{N}
%EndExpansion
\right)  $ is of strong type $\left(  L_{p_{1}}\times L_{p_{2}}\times
\cdots \times L_{p_{m}},L_{q}\right)  $ and weak type $\left(  L_{p_{1}}\times
L_{p_{2}}\times \cdots \times L_{p_{m}},L_{q,\infty}\right)  $.{ }If we take
$m=1$, $I_{\alpha}^{\left(  m\right)  }$ $\left(  m\in%
%TCIMACRO{\U{2115} }%
%BeginExpansion
\mathbb{N}
%EndExpansion
\right)  $ is the classical fractional integral operator $I_{\alpha}$.
{Moreover, their's main result (Theorem 1 in \cite{Kenig})} is the
multi-version of well-known Hardy-Littlewood-Sobolev inequality. Later,
weighted inequalities for the multilinear fractional integral operators have
been established by Moen \cite{Moen} and Chen-Xue \cite{Chen}, respectively.
Yu and Tao \cite{Yu} have also obtained the boundedness of the operators
$I_{\alpha}^{\left(  m\right)  }$, $T^{\left(  m\right)  }$ and $M^{\left(
m\right)  }$ $\left(  m\in%
%TCIMACRO{\U{2115} }%
%BeginExpansion
\mathbb{N}
%EndExpansion
\right)  $ on the product generalized Morrey spaces, respectively. Indeed
their results {(Theorem 2.1., Theorem 3.1. and Theorem 4.1. in \cite{Yu}) }are
the extensions of Theorem 4.5., Corollary 4.6., Theorem 5.4. and Corollary
5.5. in \cite{Guliyev}.

Now, we will give some properties related to the space of functions of Bounded
Mean Oscillation, $BMO$, which play a great role in the proofs of our main
results, introduced by John and Nirenberg \cite{John-Nirenberg} in 1961. This
space has become extremely important in various areas of analysis including
harmonic analysis, PDEs and function theory. $BMO$-spaces are also of interest
since, in the scale of Lebesgue spaces, they may be considered and appropriate
substitute for $L_{\infty}$. Appropriate in the sense that are spaces
preserved by a wide class of important operators such as the Hardy-Littlewood
maximal function, the Hilbert transform and which can be used as an end point
in interpolating $L_{p}$ spaces.

Let us recall the definition of the space of $BMO({\mathbb{R}^{n}})$.

\begin{definition}
\label{definition}\cite{John-Nirenberg, Karaman} The space $BMO({\mathbb{R}%
^{n}})$ of functions of bounded mean oscillation consists of locally summable
functions with finite semi-norm%
\begin{equation}
\Vert b\Vert_{\ast}\equiv \Vert b\Vert_{BMO}=\sup_{x\in{\mathbb{R}^{n}}%
,r>0}\frac{1}{|B(x,r)|}%
%TCIMACRO{\dint \limits_{B(x,r)}}%
%BeginExpansion
{\displaystyle \int \limits_{B(x,r)}}
%EndExpansion
|b(y)-b_{B(x,r)}|dy<\infty, \label{0*}%
\end{equation}
where $b_{B(x,r)}$ is the mean value of the function $b$ on the ball $B(x,r)$.
The fact that precisely the mean value $b_{B(x,r)}$ figures in (\ref{0*}) is
inessential and one gets an equivalent seminorm if $b_{B(x,r)}$ is replaced by
an arbitrary constant $c:$%
\[
\Vert b\Vert_{\ast}\approx \sup_{r>0}\inf_{c\in%
%TCIMACRO{\U{2102} }%
%BeginExpansion
\mathbb{C}
%EndExpansion
}\frac{1}{\left \vert B\left(  x,r\right)  \right \vert }\int \limits_{B\left(
x,r\right)  }\left \vert b\left(  y\right)  -c\right \vert dy.
\]
Each bounded function $b\in BMO$. Moreover, $BMO$ contains unbounded
functions, in fact log$\left \vert x\right \vert $ belongs to $BMO$ but is not
bounded, so $L_{\infty}({\mathbb{R}^{n}})\subset BMO({\mathbb{R}^{n}})$.
\end{definition}

In 1961 John and Nirenberg \cite{John-Nirenberg} established the following
deep property of functions from $BMO$.

\begin{theorem}
\label{Theorem}\cite{John-Nirenberg} If $b\in BMO({\mathbb{R}^{n}})$ and
$B\left(  x,r\right)  $ is a ball, then%
\[
\left \vert \left \{  x\in B\left(  x,r\right)  \,:\,|b(x)-b_{B\left(
x,r\right)  }|>\xi \right \}  \right \vert \leq|B\left(  x,r\right)  |\exp \left(
-\frac{\xi}{C\Vert b\Vert_{\ast}}\right)  ,~~~\xi>0,
\]
where $C$ depends only on the dimension $n$.
\end{theorem}

By Theorem \ref{Theorem}, we can get the following results.

\begin{corollary}
\cite{John-Nirenberg, Karaman} Let $b\in BMO({\mathbb{R}^{n}})$. Then, for any
$q>1$,%
\begin{equation}
\Vert b\Vert_{\ast}\thickapprox \sup_{x\in{\mathbb{R}^{n}},r>0}\left(  \frac
{1}{|B(x,r)|}%
%TCIMACRO{\dint \limits_{B(x,r)}}%
%BeginExpansion
{\displaystyle \int \limits_{B(x,r)}}
%EndExpansion
|b(y)-b_{B(x,r)}|^{p}dy\right)  ^{\frac{1}{p}} \label{5.1}%
\end{equation}
is valid.
\end{corollary}

\begin{corollary}
\label{Corollary}\cite{John-Nirenberg, Karaman} Let $b\in BMO({\mathbb{R}^{n}%
})$. Then there is a constant $C>0$ such that
\begin{equation}
\left \vert b_{B(x,r)}-b_{B(x,t)}\right \vert \leq C\Vert b\Vert_{\ast}\left(
1+\ln \frac{t}{r}\right)  \text{ }~\text{for}~0<2r<t, \label{5.2}%
\end{equation}
and for any $q>1$, it is easy to see that%
\begin{equation}
\left \Vert b-\left(  b\right)  _{B}\right \Vert _{L_{q}\left(  B\right)  }\leq
Cr^{\frac{n}{q}}\Vert b\Vert_{\ast}\left(  1+\ln \frac{t}{r}\right)  .
\label{c}%
\end{equation}
where $C$ is independent of $b$, $x$, $r$ and $t$.
\end{corollary}

Now inspired by Definition \ref{definition}, we can give the definition of
multilinear $BMO$ ($=\mathfrak{BMO}$). Indeed in this paper we will consider a
multilinear version ($=$ multilinear $BMO$ or $\mathfrak{BMO}$) of the $BMO$.

\begin{definition}
We say that $\overrightarrow{b}=\left(  b_{1},\ldots,b_{m}\right)
\in \mathfrak{BMO}$ if%
\[
\left \Vert \overrightarrow{b}\right \Vert _{\mathfrak{BMO}}=\sup_{x\in
{\mathbb{R}^{n}},r>0}%
%TCIMACRO{\dprod \limits_{i=1}^{m}}%
%BeginExpansion
{\displaystyle \prod \limits_{i=1}^{m}}
%EndExpansion
\frac{1}{|B(x,r)|}%
%TCIMACRO{\dint \limits_{B(x,r)}}%
%BeginExpansion
{\displaystyle \int \limits_{B(x,r)}}
%EndExpansion
\left \vert b_{i}\left(  y_{i}\right)  -\left(  b_{i}\right)  _{B(x,r)}%
\right \vert dy_{i}<\infty,
\]
where
\[
\left(  b_{i}\right)  _{B(x,r)}=\frac{1}{|B(x,r)|}%
%TCIMACRO{\dint \limits_{B(x,r)}}%
%BeginExpansion
{\displaystyle \int \limits_{B(x,r)}}
%EndExpansion
b_{i}(y_{i})dy_{i}.
\]

\end{definition}

\begin{remark}
Notice that $\left(  BMO\right)  ^{m}$ is contained in $\mathfrak{BMO}$ and we
have
\[
\left \Vert \overrightarrow{b}\right \Vert _{\mathfrak{BMO}}\leq%
%TCIMACRO{\dprod \limits_{i=1}^{m}}%
%BeginExpansion
{\displaystyle \prod \limits_{i=1}^{m}}
%EndExpansion
\left \Vert b_{i}\right \Vert _{\ast},
\]
so
\[
\left(  BMO\right)  ^{m}\subset \mathfrak{BMO}%
\]
is valid.
\end{remark}

We now make some conventions. Throughout this paper, we use the symbol
$A\lesssim B$ to denote that there exists a positive consant $C$ such that
$A\leq CB$. If $A\lesssim B$ and $B\lesssim A$, we then write $A\approx B$ and
say that $A$ and $B$ are equivalent. For a fixed $p\in \left[  1,\infty \right)
$, $p^{\prime}$ denotes the dual or conjugate exponent of $p$, namely,
$p^{\prime}=\frac{p}{p-1}$ and we use the convention $1^{\prime}=\infty$ and
$\infty^{\prime}=1$.

\begin{remark}
Let $0<\alpha<mn$ and $1<p_{i}<\infty$ with $\frac{1}{p}=%
%TCIMACRO{\dsum \limits_{i=1}^{m}}%
%BeginExpansion
{\displaystyle \sum \limits_{i=1}^{m}}
%EndExpansion
\frac{1}{p_{i}}$, $\frac{1}{q_{i}}=\frac{1}{p_{i}}-\frac{\alpha}{mn}$,
$\frac{1}{q}=%
%TCIMACRO{\dsum \limits_{i=1}^{m}}%
%BeginExpansion
{\displaystyle \sum \limits_{i=1}^{m}}
%EndExpansion
\frac{1}{q_{i}}=\frac{1}{p}-\frac{\alpha}{n}$ and{ }$\overrightarrow
{b}=\left(  b_{1},\ldots,b_{m}\right)  \in \left(  BMO\right)  ^{m}$ for
$i=1,\ldots,m$. Then, from Corollary \ref{Corollary}, it is easy to see that%
\begin{equation}%
%TCIMACRO{\dprod \limits_{i=1}^{m}}%
%BeginExpansion
{\displaystyle \prod \limits_{i=1}^{m}}
%EndExpansion
\left \Vert b_{i}-\left(  b_{i}\right)  _{B}\right \Vert _{L_{q_{i}}\left(
B\right)  }\leq C%
%TCIMACRO{\dprod \limits_{i=1}^{m}}%
%BeginExpansion
{\displaystyle \prod \limits_{i=1}^{m}}
%EndExpansion
|B(x,r)|^{\frac{1}{q_{i}}}\left \Vert b_{i}\right \Vert _{\ast}\left(
1+\ln \frac{t}{r}\right)  , \label{1*}%
\end{equation}
and%
\begin{align}%
%TCIMACRO{\dprod \limits_{i=1}^{m}}%
%BeginExpansion
{\displaystyle \prod \limits_{i=1}^{m}}
%EndExpansion
\left \Vert b_{i}-\left(  b_{i}\right)  _{B}\right \Vert _{L_{q_{i}}\left(
2B\right)  }  &  \leq%
%TCIMACRO{\dprod \limits_{i=1}^{m}}%
%BeginExpansion
{\displaystyle \prod \limits_{i=1}^{m}}
%EndExpansion
\left(  \left \Vert b_{i}-\left(  b_{i}\right)  _{2B}\right \Vert _{L_{q_{i}%
}\left(  2B\right)  }+\left \Vert \left(  b_{i}\right)  _{B}-\left(
b_{i}\right)  _{2B}\right \Vert _{L_{q_{i}}\left(  2B\right)  }\right)
\nonumber \\
&  \lesssim%
%TCIMACRO{\dprod \limits_{i=1}^{m}}%
%BeginExpansion
{\displaystyle \prod \limits_{i=1}^{m}}
%EndExpansion
|B(x,r)|^{\frac{1}{q_{i}}}\left \Vert b_{i}\right \Vert _{\ast}\left(
1+\ln \frac{t}{r}\right)  . \label{3*}%
\end{align}

\end{remark}

On the other hand, Xu \cite{Xu} has established the boundedness of the
commutators generated by $m$-linear {Calder\'{o}n-Zygmund }singular integrals
and $RBMO$ functions with nonhomogeneity on the{\ product of Lebesgue space}.
Inspired by \cite{Chen1, Chen, Grafakos1, Grafakos3*, Si, Xu}, commutators
$T_{\overrightarrow{b}}^{\left(  m\right)  }$ generated by $m$-linear
{Calder\'{o}n-Zygmund operators }$T^{\left(  m\right)  }$ {and bounded mean
oscillation functions }$\overrightarrow{b}=\left(  b_{1},\ldots,b_{m}\right)
${\ is given by}%
\[
T_{\overrightarrow{b}}^{\left(  m\right)  }\left(  \overrightarrow{f}\right)
\left(  x\right)  =%
%TCIMACRO{\dint \limits_{\left(  {\mathbb{R}^{n}}\right)  ^{m}}}%
%BeginExpansion
{\displaystyle \int \limits_{\left(  {\mathbb{R}^{n}}\right)  ^{m}}}
%EndExpansion
K\left(  x,y_{1},\ldots,y_{m}\right)  \left[
%TCIMACRO{\dprod \limits_{i=1}^{m}}%
%BeginExpansion
{\displaystyle \prod \limits_{i=1}^{m}}
%EndExpansion
\left[  b_{i}\left(  x\right)  -b_{i}\left(  y_{i}\right)  \right]
f_{i}\left(  y_{i}\right)  \right]  d\overrightarrow{y},
\]
where $K\left(  x,y_{1},\ldots,y_{m}\right)  $ is a $m$-linear
{Calder\'{o}n-Zygmund kernel, }$b_{i}\in \left(  BMO\right)  ^{i}%
({\mathbb{R}^{n}})$ for $i=1,\ldots,m$. Note that $T_{b}$ is the special case
of $T_{\overrightarrow{b}}^{\left(  m\right)  }$ with taking $m=1$. Similarly,
let $b_{i}\left(  i=1,\ldots,m\right)  $ be a locally integrable functions on
${\mathbb{R}^{n}}$, then the commutators generated by $m$-linear fractional
integral operators and $\overrightarrow{b}=\left(  b_{1},\ldots,b_{m}\right)
$ is given by%
\[
I_{\alpha,\overrightarrow{b}}^{\left(  m\right)  }\left(  \overrightarrow
{f}\right)  \left(  x\right)  =%
%TCIMACRO{\dint \limits_{\left(  {\mathbb{R}^{n}}\right)  ^{m}}}%
%BeginExpansion
{\displaystyle \int \limits_{\left(  {\mathbb{R}^{n}}\right)  ^{m}}}
%EndExpansion
\frac{1}{\left \vert \left(  x-y_{1},\ldots,x-y_{m}\right)  \right \vert
^{mn-\alpha}}\left[
%TCIMACRO{\dprod \limits_{i=1}^{m}}%
%BeginExpansion
{\displaystyle \prod \limits_{i=1}^{m}}
%EndExpansion
\left[  b_{i}\left(  x\right)  -b_{i}\left(  y_{i}\right)  \right]
f_{i}\left(  y_{i}\right)  \right]  d\overrightarrow{y},
\]
where $0<\alpha<mn$, and $f_{i}$ $\left(  i=1,\ldots,m\right)  $ are suitable functions.

The commutators of a class of multi-sublinear maximal operators corresponding
to $T_{\overrightarrow{b}}^{\left(  m\right)  }$ and $I_{\alpha
,\overrightarrow{b}}^{\left(  m\right)  }$ $\left(  m\in%
%TCIMACRO{\U{2115} }%
%BeginExpansion
\mathbb{N}
%EndExpansion
\right)  $ above are, respectively, defined by%
\[
M_{\overrightarrow{b}}^{\left(  m\right)  }\left(  \overrightarrow{f}\right)
\left(  x\right)  =\sup_{t>0}\left[
%TCIMACRO{\dprod \limits_{i=1}^{m}}%
%BeginExpansion
{\displaystyle \prod \limits_{i=1}^{m}}
%EndExpansion
\frac{1}{\left \vert B\left(  x,t\right)  \right \vert }\int \limits_{B\left(
x,t\right)  }\left[  \left \vert b_{i}\left(  x\right)  -b_{i}\left(
y_{i}\right)  \right \vert \right]  \left \vert f_{i}\left(  y_{i}\right)
\right \vert \right]  d\overrightarrow{y},
\]
and
\[
M_{\alpha,\overrightarrow{b}}^{\left(  m\right)  }\left(  \overrightarrow
{f}\right)  \left(  x\right)  =\sup_{t>0}\left \vert B\left(  x,t\right)
\right \vert ^{\frac{\alpha}{n}}\left[
%TCIMACRO{\dprod \limits_{i=1}^{m}}%
%BeginExpansion
{\displaystyle \prod \limits_{i=1}^{m}}
%EndExpansion
\frac{1}{\left \vert B\left(  x,t\right)  \right \vert }\int \limits_{B\left(
x,t\right)  }\left[  \left \vert b_{i}\left(  x\right)  -b_{i}\left(
y_{i}\right)  \right \vert \right]  \left \vert f_{i}\left(  y_{i}\right)
\right \vert \right]  d\overrightarrow{y},\text{ \  \ }0\leq \alpha<mn.
\]

The following result is known.

\begin{lemma}
\label{Lemma0}\cite{Si} (Strong bounds of $I_{\overrightarrow{b},\alpha
}^{\left(  m\right)  }$) Let $0<\alpha_{i}<n$, $1<p_{1},\ldots,p_{m}<\infty$,
$\frac{1}{p}=%
%TCIMACRO{\dsum \limits_{i=1}^{m}}%
%BeginExpansion
{\displaystyle \sum \limits_{i=1}^{m}}
%EndExpansion
\frac{1}{p_{i}}$, $\alpha=%
%TCIMACRO{\dsum \limits_{i=1}^{m}}%
%BeginExpansion
{\displaystyle \sum \limits_{i=1}^{m}}
%EndExpansion
\alpha_{i}$ and $\frac{1}{q}=\frac{1}{p}-\frac{\alpha}{n}$. Then there is
$C>0$ independent of $\overrightarrow{f}$ and $\overrightarrow{b}$ such that%
\[
\left \Vert I_{\overrightarrow{b},\alpha}^{\left(  m\right)  }\left(
\overrightarrow{f}\right)  \right \Vert _{L_{q}\left(  {\mathbb{R}^{n}}\right)
}\leq C%
%TCIMACRO{\dprod \limits_{i=1}^{m}}%
%BeginExpansion
{\displaystyle \prod \limits_{i=1}^{m}}
%EndExpansion
\Vert b_{i}\Vert_{\ast}\left \Vert f_{i}\right \Vert _{L_{p_{i}}\left(
{\mathbb{R}^{n}}\right)  }.
\]

\end{lemma}

Using the idea in the proof of Lemma 3.2 in \cite{Gurbuz4}, we can obtain the
following Corollary \ref{Corollary0*}:

\begin{corollary}
\label{Corollary0*}(Strong bounds of $M_{\alpha,\overrightarrow{b}}^{\left(
m\right)  }$) Under the assumptions of Lemma \ref{Lemma0}, the operator
$M_{\alpha,\overrightarrow{b}}^{\left(  m\right)  }$ is bounded from
$L_{p_{1}}({\mathbb{R}^{n}})\times \cdots L_{p_{m}}({\mathbb{R}^{n}})$ to
$L_{q}({\mathbb{R}^{n}})$. Moreover, we have%
\[
\left \Vert M_{\alpha,\overrightarrow{b}}^{\left(  m\right)  }\left(
\overrightarrow{f}\right)  \right \Vert _{L_{q}\left(  {\mathbb{R}^{n}}\right)
}\leq C%
%TCIMACRO{\dprod \limits_{i=1}^{m}}%
%BeginExpansion
{\displaystyle \prod \limits_{i=1}^{m}}
%EndExpansion
\Vert b_{i}\Vert_{\ast}\left \Vert f_{i}\right \Vert _{L_{p_{i}}\left(
{\mathbb{R}^{n}}\right)  }.
\]

\end{corollary}

\begin{proof}
Set%
\[
\widetilde{I}_{\overrightarrow{b},\alpha}^{\left(  m\right)  }\left(
\left \vert f\right \vert \right)  (x)=%
%TCIMACRO{\dint \limits_{\left(  {\mathbb{R}^{n}}\right)  ^{m}}}%
%BeginExpansion
{\displaystyle \int \limits_{\left(  {\mathbb{R}^{n}}\right)  ^{m}}}
%EndExpansion
\frac{1}{\left \vert \left(  x-y_{1},\ldots,x-y_{m}\right)  \right \vert
^{mn-\alpha}}\left[
%TCIMACRO{\dprod \limits_{i=1}^{m}}%
%BeginExpansion
{\displaystyle \prod \limits_{i=1}^{m}}
%EndExpansion
\left[  \left \vert b_{i}\left(  x\right)  -b_{i}\left(  y_{i}\right)
\right \vert \right]  \left \vert f_{i}\left(  y_{i}\right)  \right \vert
\right]  d\overrightarrow{y}\qquad0<\alpha<mn.
\]
It is easy to see that Lemma \ref{Lemma0} is also hold for $\widetilde
{I}_{\overrightarrow{b},\alpha}^{\left(  m\right)  }$. On the other hand, for
any $t>0$, we have
\begin{align*}
\widetilde{I}_{\overrightarrow{b},\alpha}^{\left(  m\right)  }\left(
\left \vert f\right \vert \right)  (x)  &  \geq%
%TCIMACRO{\dint \limits_{\left(  B\left(  x,t\right)  \right)  ^{m}}}%
%BeginExpansion
{\displaystyle \int \limits_{\left(  B\left(  x,t\right)  \right)  ^{m}}}
%EndExpansion
\frac{1}{\left \vert \left(  x-y_{1},\ldots,x-y_{m}\right)  \right \vert
^{mn-\alpha}}\left[
%TCIMACRO{\dprod \limits_{i=1}^{m}}%
%BeginExpansion
{\displaystyle \prod \limits_{i=1}^{m}}
%EndExpansion
\left[  \left \vert b_{i}\left(  x\right)  -b_{i}\left(  y_{i}\right)
\right \vert \right]  \left \vert f_{i}\left(  y_{i}\right)  \right \vert
\right]  d\overrightarrow{y}\\
&  \geq \frac{1}{t^{mn-\alpha}}%
%TCIMACRO{\dint \limits_{B\left(  x,t\right)  }}%
%BeginExpansion
{\displaystyle \int \limits_{B\left(  x,t\right)  }}
%EndExpansion
\left[
%TCIMACRO{\dprod \limits_{i=1}^{m}}%
%BeginExpansion
{\displaystyle \prod \limits_{i=1}^{m}}
%EndExpansion
\left[  \left \vert b_{i}\left(  x\right)  -b_{i}\left(  y_{i}\right)
\right \vert \right]  \left \vert f_{i}\left(  y_{i}\right)  \right \vert
\right]  d\overrightarrow{y}.
\end{align*}
Taking supremum over $t>0$ in the above inequality, we get%
\begin{equation}
M_{\alpha,\overrightarrow{b}}^{\left(  m\right)  }\left(  \overrightarrow
{f}\right)  \left(  x\right)  \leq C_{n,\alpha}^{-1}\widetilde{I}%
_{\overrightarrow{b},\alpha}^{\left(  m\right)  }\left(  \left \vert
f\right \vert \right)  (x)\qquad C_{n,\alpha}=\left \vert B\left(  0,1\right)
\right \vert ^{\frac{mn-\alpha}{n}}. \label{4}%
\end{equation}

\end{proof}

As a simple corollary of Lemma \ref{Lemma0} and Corollary \ref{Corollary0*},
we can obtain the following result.

\begin{corollary}
(Strong bounds of $T_{\overrightarrow{b}}^{\left(  m\right)  }$ and
$M_{\overrightarrow{b}}^{\left(  m\right)  }$) Let $1<p_{1},\ldots
,p_{m}<\infty$ and $0<p<\infty$ with $\frac{1}{p}=%
%TCIMACRO{\dsum \limits_{i=1}^{m}}%
%BeginExpansion
{\displaystyle \sum \limits_{i=1}^{m}}
%EndExpansion
\frac{1}{p_{i}}$. Then there is $C>0$ independent of $\overrightarrow{f}$ and
$\overrightarrow{b}$ such that%
\[
\left \Vert T_{\overrightarrow{b}}^{\left(  m\right)  }\left(  \overrightarrow
{f}\right)  \right \Vert _{L_{p}\left(  {\mathbb{R}^{n}}\right)  }\leq C%
%TCIMACRO{\dprod \limits_{i=1}^{m}}%
%BeginExpansion
{\displaystyle \prod \limits_{i=1}^{m}}
%EndExpansion
\Vert b_{i}\Vert_{\ast}\left \Vert f_{i}\right \Vert _{L_{p_{i}}\left(
{\mathbb{R}^{n}}\right)  },
\]%
\[
\left \Vert M_{\overrightarrow{b}}^{\left(  m\right)  }\left(  \overrightarrow
{f}\right)  \right \Vert _{L_{p}\left(  {\mathbb{R}^{n}}\right)  }\leq C%
%TCIMACRO{\dprod \limits_{i=1}^{m}}%
%BeginExpansion
{\displaystyle \prod \limits_{i=1}^{m}}
%EndExpansion
\Vert b_{i}\Vert_{\ast}\left \Vert f_{i}\right \Vert _{L_{p_{i}}\left(
{\mathbb{R}^{n}}\right)  }.
\]

\end{corollary}

The purpose of this paper is to consider the mapping properties on
$M_{p_{1},\varphi_{1}}\times \cdots \times M_{p_{m},\varphi_{m}}$ and
$VM_{p_{1},\varphi_{1}}\times \cdots \times VM_{p_{m},\varphi_{m}}$ for the
commutators of multilinear fractional maximal and integral{\ operators},
respectively. Similar results still hold for commutators of multilinear
maximal and singular integral{\ operators. }Commutators of multilinear
fractional maximal and integral{\ operators on} product generalized Morrey
spaces have not also been studied so far and this paper seems to be the first
in this direction. Now, let us state the main results of this paper. Indeed
our final result is the following theorem, which is an extension of Theorem
7.4. and Corollary 7.5. in \cite{Guliyev}.

\begin{theorem}
\label{teo1} Let $0<\alpha<mn$ and $1\leq p_{i}<\frac{mn}{\alpha}$ with
$\frac{1}{p}=%
%TCIMACRO{\dsum \limits_{i=1}^{m}}%
%BeginExpansion
{\displaystyle \sum \limits_{i=1}^{m}}
%EndExpansion
\frac{1}{p_{i}}$, $\frac{1}{q}=\sum \limits_{i=1}^{m}\frac{1}{p_{i}}%
+\sum \limits_{i=1}^{m}\frac{1}{q_{i}}-\frac{\alpha}{n}$ and $\overrightarrow
{b}\in \left(  BMO\right)  ^{m}({\mathbb{R}^{n}})$ for $i=1,\ldots,m$. Let
functions $\varphi,\varphi_{i}:{\mathbb{R}^{n}\times}\left(  0,\infty \right)
\rightarrow \left(  0,\infty \right)  $ $\left(  i=1,\ldots,m\right)  $ and
$(\varphi_{1},\ldots,\varphi_{m},\varphi)$ satisfies the condition%
\begin{equation}%
%TCIMACRO{\dint \limits_{r}^{\infty}}%
%BeginExpansion
{\displaystyle \int \limits_{r}^{\infty}}
%EndExpansion
\left(  1+\ln \frac{t}{r}\right)  ^{m}\frac{\operatorname*{essinf}%
\limits_{t<\tau<\infty}%
%TCIMACRO{\dprod \limits_{i=1}^{m}}%
%BeginExpansion
{\displaystyle \prod \limits_{i=1}^{m}}
%EndExpansion
\varphi_{i}(x,\tau)\tau^{\frac{n}{p}}}{t^{n\left(  \frac{1}{q}-%
%TCIMACRO{\dsum \limits_{i=1}^{m}}%
%BeginExpansion
{\displaystyle \sum \limits_{i=1}^{m}}
%EndExpansion
\frac{1}{q_{i}}\right)  +1}}dt\leq C\varphi \left(  x,r\right)  , \label{50}%
\end{equation}
where $C$ does not depend on $x\in{\mathbb{R}^{n}}$ and $r>0$.

Then, $I_{\alpha,\overrightarrow{b}}^{\left(  m\right)  }$ and $M_{\alpha
,\overrightarrow{b}}^{\left(  m\right)  }$ $\left(  m\in%
%TCIMACRO{\U{2115} }%
%BeginExpansion
\mathbb{N}
%EndExpansion
\right)  $ are bounded operators from product space $M_{p_{1},\varphi_{1}%
}\times \cdots \times M_{p_{m},\varphi_{m}}$ to $M_{q,\varphi}$. Moreover, we
have
\begin{equation}
\left \Vert I_{\alpha,\overrightarrow{b}}^{\left(  m\right)  }\left(
\overrightarrow{f}\right)  \right \Vert _{M_{q,\varphi}}\lesssim \left \Vert
\overrightarrow{b}\right \Vert _{\mathfrak{BMO}}\left \Vert f_{i}\right \Vert
_{M_{p_{i},\varphi_{i}}}\lesssim%
%TCIMACRO{\dprod \limits_{i=1}^{m}}%
%BeginExpansion
{\displaystyle \prod \limits_{i=1}^{m}}
%EndExpansion
\Vert b_{i}\Vert_{\ast}\left \Vert f_{i}\right \Vert _{M_{p_{i},\varphi_{i}}},
\label{51}%
\end{equation}%
\begin{equation}
\left \Vert M_{\alpha,\overrightarrow{b}}^{\left(  m\right)  }\left(
\overrightarrow{f}\right)  \right \Vert _{M_{q,\varphi}}\lesssim \left \Vert
\overrightarrow{b}\right \Vert _{\mathfrak{BMO}}\left \Vert f_{i}\right \Vert
_{M_{p_{i},\varphi_{i}}}\lesssim%
%TCIMACRO{\dprod \limits_{i=1}^{m}}%
%BeginExpansion
{\displaystyle \prod \limits_{i=1}^{m}}
%EndExpansion
\Vert b_{i}\Vert_{\ast}\left \Vert f_{i}\right \Vert _{M_{p_{i},\varphi_{i}}}.
\label{52}%
\end{equation}

\end{theorem}

\begin{corollary}
Let $1<p_{i}<\infty$ and $0<p<\infty$ with $\frac{1}{p}=%
%TCIMACRO{\dsum \limits_{i=1}^{m}}%
%BeginExpansion
{\displaystyle \sum \limits_{i=1}^{m}}
%EndExpansion
\frac{1}{p_{i}}$ and $\overrightarrow{b}\in \left(  BMO\right)  ^{m}%
({\mathbb{R}^{n}})$ for $i=1,\ldots,m$. Let functions $\varphi,\varphi
_{i}:{\mathbb{R}^{n}\times}\left(  0,\infty \right)  \rightarrow \left(
0,\infty \right)  $ $\left(  i=1,\ldots,m\right)  $ and $(\varphi_{1}%
,\ldots,\varphi_{m},\varphi)$ satisfies the condition%
\[%
%TCIMACRO{\dint \limits_{r}^{\infty}}%
%BeginExpansion
{\displaystyle \int \limits_{r}^{\infty}}
%EndExpansion
\left(  1+\ln \frac{t}{r}\right)  ^{m}\frac{\operatorname*{essinf}%
\limits_{t<\tau<\infty}%
%TCIMACRO{\dprod \limits_{i=1}^{m}}%
%BeginExpansion
{\displaystyle \prod \limits_{i=1}^{m}}
%EndExpansion
\varphi_{i}(x,\tau)\tau^{\frac{n}{p}}}{t^{\frac{n}{p}+1}}dt\leq C\varphi
\left(  x,r\right)  ,
\]
where $C$ does not depend on $x\in{\mathbb{R}^{n}}$ and $r>0$.

Then, $T_{\overrightarrow{b}}^{\left(  m\right)  }$ and $M_{\overrightarrow
{b}}^{\left(  m\right)  }$ $\left(  m\in%
%TCIMACRO{\U{2115} }%
%BeginExpansion
\mathbb{N}
%EndExpansion
\right)  $ are bounded operators from product space $M_{p_{1},\varphi_{1}%
}\times \cdots \times M_{p_{m},\varphi_{m}}$ to $M_{p,\varphi}$. Moreover, we
have
\[
\left \Vert T_{\overrightarrow{b}}^{\left(  m\right)  }\left(  \overrightarrow
{f}\right)  \right \Vert _{M_{p,\varphi}}\lesssim \left \Vert \overrightarrow
{b}\right \Vert _{\mathfrak{BMO}}\left \Vert f_{i}\right \Vert _{M_{p_{i}%
,\varphi_{i}}}\lesssim%
%TCIMACRO{\dprod \limits_{i=1}^{m}}%
%BeginExpansion
{\displaystyle \prod \limits_{i=1}^{m}}
%EndExpansion
\Vert b_{i}\Vert_{\ast}\left \Vert f_{i}\right \Vert _{M_{p_{i},\varphi_{i}}},
\]%
\[
\left \Vert M_{\overrightarrow{b}}^{\left(  m\right)  }\left(  \overrightarrow
{f}\right)  \right \Vert _{M_{p,\varphi}}\lesssim \left \Vert \overrightarrow
{b}\right \Vert _{\mathfrak{BMO}}\left \Vert f_{i}\right \Vert _{M_{p_{i}%
,\varphi_{i}}}\lesssim%
%TCIMACRO{\dprod \limits_{i=1}^{m}}%
%BeginExpansion
{\displaystyle \prod \limits_{i=1}^{m}}
%EndExpansion
\Vert b_{i}\Vert_{\ast}\left \Vert f_{i}\right \Vert _{M_{p_{i},\varphi_{i}}}.
\]

\end{corollary}

Our another main result is the following.

\begin{theorem}
\label{teo2}Let $0<\alpha<mn$ and $1\leq p_{i}<\frac{mn}{\alpha}$ with
$\frac{1}{p}=%
%TCIMACRO{\dsum \limits_{i=1}^{m}}%
%BeginExpansion
{\displaystyle \sum \limits_{i=1}^{m}}
%EndExpansion
\frac{1}{p_{i}}$, $\frac{1}{q}=\sum \limits_{i=1}^{m}\frac{1}{p_{i}}%
+\sum \limits_{i=1}^{m}\frac{1}{q_{i}}-\frac{\alpha}{n}$ and $\overrightarrow
{b}\in \left(  BMO\right)  ^{m}({\mathbb{R}^{n}})$ for $i=1,\ldots,m$. Let
functions $\varphi,\varphi_{i}:{\mathbb{R}^{n}\times}\left(  0,\infty \right)
\rightarrow \left(  0,\infty \right)  $ $\left(  i=1,\ldots,m\right)  $ and
$(\varphi_{1},\ldots,\varphi_{m},\varphi)$ satisfies conditions (\ref{2}%
)-(\ref{3}) and%
\begin{equation}%
%TCIMACRO{\dint \limits_{r}^{\infty}}%
%BeginExpansion
{\displaystyle \int \limits_{r}^{\infty}}
%EndExpansion
\left(  1+\ln \frac{t}{r}\right)  ^{m}%
%TCIMACRO{\dprod \limits_{i=1}^{m}}%
%BeginExpansion
{\displaystyle \prod \limits_{i=1}^{m}}
%EndExpansion
\varphi_{i}(x,t)\frac{t^{\frac{n}{p}}}{t^{n\left(  \frac{1}{q}-%
%TCIMACRO{\dsum \limits_{i=1}^{m}}%
%BeginExpansion
{\displaystyle \sum \limits_{i=1}^{m}}
%EndExpansion
\frac{1}{q_{i}}\right)  +1}}dt\leq C_{0}\varphi \left(  x,r\right)  ,
\label{10*}%
\end{equation}
where $C_{0}$ does not depend on $x\in{\mathbb{R}^{n}}$ and $r>0$,%
\begin{equation}
\lim_{r\rightarrow0}\frac{\ln \frac{1}{r}}{\inf \limits_{x\in{\mathbb{R}^{n}}%
}\varphi(x,r)}=0 \label{11*}%
\end{equation}
and
\begin{equation}
c_{\delta}:=%
%TCIMACRO{\dint \limits_{\delta}^{\infty}}%
%BeginExpansion
{\displaystyle \int \limits_{\delta}^{\infty}}
%EndExpansion
\left(  1+\ln \left \vert t\right \vert \right)  ^{m}\sup \limits_{x\in
{\mathbb{R}^{n}}}%
%TCIMACRO{\dprod \limits_{i=1}^{m}}%
%BeginExpansion
{\displaystyle \prod \limits_{i=1}^{m}}
%EndExpansion
\varphi_{i}(x,t)\frac{t^{\frac{n}{p}}}{t^{n\left(  \frac{1}{q}-%
%TCIMACRO{\dsum \limits_{i=1}^{m}}%
%BeginExpansion
{\displaystyle \sum \limits_{i=1}^{m}}
%EndExpansion
\frac{1}{q_{i}}\right)  +1}}dt<\infty \label{12*}%
\end{equation}
for every $\delta>0$.

Then, $I_{\alpha,\overrightarrow{b}}^{\left(  m\right)  }$ and $M_{\alpha
,\overrightarrow{b}}^{\left(  m\right)  }$ $\left(  m\in%
%TCIMACRO{\U{2115} }%
%BeginExpansion
\mathbb{N}
%EndExpansion
\right)  $ are bounded operators from product space $VM_{p_{1},\varphi_{1}%
}\times \cdots \times VM_{p_{m},\varphi_{m}}$ to $VM_{q,\varphi}$. Moreover, we
have
\begin{equation}
\left \Vert I_{\alpha,\overrightarrow{b}}^{\left(  m\right)  }\left(
\overrightarrow{f}\right)  \right \Vert _{VM_{q,\varphi}}\lesssim \left \Vert
\overrightarrow{b}\right \Vert _{\mathfrak{BMO}}\left \Vert f_{i}\right \Vert
_{VM_{p_{i},\varphi_{i}}}\lesssim%
%TCIMACRO{\dprod \limits_{i=1}^{m}}%
%BeginExpansion
{\displaystyle \prod \limits_{i=1}^{m}}
%EndExpansion
\Vert b_{i}\Vert_{\ast}\left \Vert f_{i}\right \Vert _{VM_{p_{i},\varphi_{i}}},
\label{53}%
\end{equation}%
\begin{equation}
\left \Vert M_{\alpha,\overrightarrow{b}}^{\left(  m\right)  }\left(
\overrightarrow{f}\right)  \right \Vert _{VM_{q,\varphi}}\lesssim \left \Vert
\overrightarrow{b}\right \Vert _{\mathfrak{BMO}}\left \Vert f_{i}\right \Vert
_{VM_{p_{i},\varphi_{i}}}\lesssim%
%TCIMACRO{\dprod \limits_{i=1}^{m}}%
%BeginExpansion
{\displaystyle \prod \limits_{i=1}^{m}}
%EndExpansion
\Vert b_{i}\Vert_{\ast}\left \Vert f_{i}\right \Vert _{VM_{p_{i},\varphi_{i}}}.
\label{54}%
\end{equation}

\end{theorem}

\begin{corollary}
Let $1<p_{i}<\infty$ and $0<p<\infty$ with $\frac{1}{p}=%
%TCIMACRO{\dsum \limits_{i=1}^{m}}%
%BeginExpansion
{\displaystyle \sum \limits_{i=1}^{m}}
%EndExpansion
\frac{1}{p_{i}}$ and $\overrightarrow{b}\in \left(  BMO\right)  ^{m}%
({\mathbb{R}^{n}})$ for $i=1,\ldots,m$. Let functions $\varphi,\varphi
_{i}:{\mathbb{R}^{n}\times}\left(  0,\infty \right)  \rightarrow \left(
0,\infty \right)  $ $\left(  i=1,\ldots,m\right)  $ and $(\varphi_{1}%
,\ldots,\varphi_{m},\varphi)$ satisfies conditions (\ref{2})-(\ref{3}) and%
\[%
%TCIMACRO{\dint \limits_{r}^{\infty}}%
%BeginExpansion
{\displaystyle \int \limits_{r}^{\infty}}
%EndExpansion
\left(  1+\ln \frac{t}{r}\right)  ^{m}%
%TCIMACRO{\dprod \limits_{i=1}^{m}}%
%BeginExpansion
{\displaystyle \prod \limits_{i=1}^{m}}
%EndExpansion
\varphi_{i}(x,t)\frac{t^{\frac{n}{p}}}{t^{\frac{n}{p}+1}}dt\leq C_{0}%
\varphi \left(  x,r\right)  ,
\]
where $C_{0}$ does not depend on $x\in{\mathbb{R}^{n}}$ and $r>0$,%
\[
\lim_{r\rightarrow0}\frac{\ln \frac{1}{r}}{\inf \limits_{x\in{\mathbb{R}^{n}}%
}\varphi(x,r)}=0
\]
and
\[
c_{\delta}:=%
%TCIMACRO{\dint \limits_{\delta}^{\infty}}%
%BeginExpansion
{\displaystyle \int \limits_{\delta}^{\infty}}
%EndExpansion
\left(  1+\ln \left \vert t\right \vert \right)  ^{m}\sup \limits_{x\in
{\mathbb{R}^{n}}}%
%TCIMACRO{\dprod \limits_{i=1}^{m}}%
%BeginExpansion
{\displaystyle \prod \limits_{i=1}^{m}}
%EndExpansion
\varphi_{i}(x,t)\frac{t^{\frac{n}{p}}}{t^{\frac{n}{p}+1}}dt<\infty
\]
for every $\delta>0$.

Then, $T_{\overrightarrow{b}}^{\left(  m\right)  }$ and $M_{\overrightarrow
{b}}^{\left(  m\right)  }$ $\left(  m\in%
%TCIMACRO{\U{2115} }%
%BeginExpansion
\mathbb{N}
%EndExpansion
\right)  $ are bounded operators from product space $VM_{p_{1},\varphi_{1}%
}\times \cdots \times VM_{p_{m},\varphi_{m}}$ to $VM_{p,\varphi}$. Moreover, we
have
\[
\left \Vert T_{\overrightarrow{b}}^{\left(  m\right)  }\left(  \overrightarrow
{f}\right)  \right \Vert _{VM_{p,\varphi}}\lesssim \left \Vert \overrightarrow
{b}\right \Vert _{\mathfrak{BMO}}\left \Vert f_{i}\right \Vert _{VM_{p_{i}%
,\varphi_{i}}}\lesssim%
%TCIMACRO{\dprod \limits_{i=1}^{m}}%
%BeginExpansion
{\displaystyle \prod \limits_{i=1}^{m}}
%EndExpansion
\Vert b_{i}\Vert_{\ast}\left \Vert f_{i}\right \Vert _{VM_{p_{i},\varphi_{i}}},
\]%
\[
\left \Vert M_{\overrightarrow{b}}^{\left(  m\right)  }\left(  \overrightarrow
{f}\right)  \right \Vert _{VM_{p,\varphi}}\lesssim \left \Vert \overrightarrow
{b}\right \Vert _{\mathfrak{BMO}}\left \Vert f_{i}\right \Vert _{VM_{p_{i}%
,\varphi_{i}}}\lesssim%
%TCIMACRO{\dprod \limits_{i=1}^{m}}%
%BeginExpansion
{\displaystyle \prod \limits_{i=1}^{m}}
%EndExpansion
\Vert b_{i}\Vert_{\ast}\left \Vert f_{i}\right \Vert _{VM_{p_{i},\varphi_{i}}}.
\]

\end{corollary}

The article is organized as follows. A key lemma is given and proved in
Section 2. Section 3 will be devoted to the proofs of the theorems (Theorems
\ref{teo1} and \ref{teo2}) stated above.

\section{A Key Lemma}

In order to prove the main results (Theorems \ref{teo1} and \ref{teo2}), we
need the following lemma.

\begin{lemma}
\label{Lemma}Let $x_{0}\in{\mathbb{R}^{n}}$, $0<\alpha<mn$ and $1\leq
p_{i}<\frac{mn}{\alpha}$ with $\frac{1}{p}=%
%TCIMACRO{\dsum \limits_{i=1}^{m}}%
%BeginExpansion
{\displaystyle \sum \limits_{i=1}^{m}}
%EndExpansion
\frac{1}{p_{i}}$, $\frac{1}{q}=\sum \limits_{i=1}^{m}\frac{1}{p_{i}}%
+\sum \limits_{i=1}^{m}\frac{1}{q_{i}}-\frac{\alpha}{n}$ and $\overrightarrow
{b}\in \left(  BMO\right)  ^{m}({\mathbb{R}^{n}})$ for $i=1,\ldots,m$. Then the
inequality%
\begin{equation}
\Vert I_{\alpha,\overrightarrow{b}}^{\left(  m\right)  }\left(
\overrightarrow{f}\right)  \Vert_{L_{q}(B\left(  x_{0},r\right)  )}\lesssim%
%TCIMACRO{\dprod \limits_{i=1}^{m}}%
%BeginExpansion
{\displaystyle \prod \limits_{i=1}^{m}}
%EndExpansion
\Vert b_{i}\Vert_{\ast}r^{\frac{n}{q}}\int \limits_{2r}^{\infty}\left(
1+\ln \frac{t}{r}\right)  ^{m}%
%TCIMACRO{\dprod \limits_{i=1}^{m}}%
%BeginExpansion
{\displaystyle \prod \limits_{i=1}^{m}}
%EndExpansion
\Vert f_{i}\Vert_{L_{p_{i}}(B\left(  x_{0},t\right)  )}\frac{dt}{t^{n\left(
\frac{1}{q}-%
%TCIMACRO{\dsum \limits_{i=1}^{m}}%
%BeginExpansion
{\displaystyle \sum \limits_{i=1}^{m}}
%EndExpansion
\frac{1}{q_{i}}\right)  +1}} \label{200}%
\end{equation}
holds for any ball $B(x_{0},r)$ and for all $\overrightarrow{f}\in L_{p_{1}%
}^{loc}\left(  {\mathbb{R}^{n}}\right)  \times \cdots \times L_{p_{m}}%
^{loc}\left(  {\mathbb{R}^{n}}\right)  $.
\end{lemma}

\begin{proof}
In order to simplify the proof, we consider only the situation when $m=2$.
Actually, a similar procedure works for all $m\in%
%TCIMACRO{\U{2115} }%
%BeginExpansion
\mathbb{N}
%EndExpansion
$. Thus, without loss of generality, it is sufficient to show that the
conclusion holds for $I_{\alpha,\overrightarrow{b}}^{\left(  2\right)
}\left(  \overrightarrow{f}\right)  =I_{\alpha,\left(  b_{1},b_{2}\right)
}^{\left(  2\right)  }\left(  f_{1},f_{2}\right)  $.

We just consider the case $p_{i}>1$ for $i=1,2$. For any $x_{0}\in
{\mathbb{R}^{n}}$, set $B=B\left(  x_{0},r\right)  $ for the ball centered at
$x_{0}$ and of radius $r$ and $2B=B\left(  x_{0},2r\right)  $. Indeed, we also
decompose $f_{i}$ as $f_{i}\left(  y_{i}\right)  =$ $f_{i}\left(
y_{i}\right)  \chi_{2B}+f_{i}\left(  y_{i}\right)  \chi_{\left(  2B\right)
^{c}}$ for $i=1,2$. And, we write $f_{1}=f_{1}^{0}+f_{1}^{\infty}$ and
$f_{2}=f_{2}^{0}+f_{2}^{\infty}$, where \ $f_{i}^{0}=f_{i}\chi_{2B}$,
\ $f_{i}^{\infty}=f_{i}\chi_{\left(  2B\right)  ^{c}}$, for $i=1,2$. Thus, we
have%
\begin{align*}
\left \Vert I_{\alpha,\left(  b_{1},b_{2}\right)  }^{\left(  2\right)  }\left(
f_{1},f_{2}\right)  \right \Vert _{L_{q}\left(  B\left(  x_{0},r\right)
\right)  }  &  \leq \left \Vert I_{\alpha,\left(  b_{1},b_{2}\right)  }^{\left(
2\right)  }\left(  f_{1}^{0},f_{2}^{0}\right)  \right \Vert _{L_{q}\left(
B\left(  x_{0},r\right)  \right)  }+\left \Vert I_{\alpha,\left(  b_{1}%
,b_{2}\right)  }^{\left(  2\right)  }\left(  f_{1}^{0},f_{2}^{\infty}\right)
\right \Vert _{L_{q}\left(  B\left(  x_{0},r\right)  \right)  }\\
&  +\left \Vert I_{\alpha,\left(  b_{1},b_{2}\right)  }^{\left(  2\right)
}\left(  f_{1}^{\infty},f_{2}^{0}\right)  \right \Vert _{L_{q}\left(  B\left(
x_{0},r\right)  \right)  }+\left \Vert I_{\alpha,\left(  b_{1},b_{2}\right)
}^{\left(  2\right)  }\left(  f_{1}^{\infty},f_{2}^{\infty}\right)
\right \Vert _{L_{q}\left(  B\left(  x_{0},r\right)  \right)  }\\
&  =F_{1}+F_{2}+F_{3}+F_{4}.
\end{align*}

Firstly, we use the boundedness of $I_{\alpha,\left(  b_{1},b_{2}\right)
}^{\left(  2\right)  }$ from $L_{p_{1}}\times L_{p_{2}}$ into $L_{q}$ (see
Lemma \ref{Lemma0}) to estimate $F_{1}$, and we obtain%
\begin{align*}
F_{1}  &  =\left \Vert I_{\alpha,\left(  b_{1},b_{2}\right)  }^{\left(
2\right)  }\left(  f_{1}^{0},f_{2}^{0}\right)  \right \Vert _{L_{q}\left(
B\left(  x_{0},r\right)  \right)  }\lesssim%
%TCIMACRO{\dprod \limits_{i=1}^{2}}%
%BeginExpansion
{\displaystyle \prod \limits_{i=1}^{2}}
%EndExpansion
\Vert b_{i}\Vert_{\ast}\left \Vert f_{i}\right \Vert _{L_{p_{i}}\left(
2B\right)  }\\
&  \lesssim r^{\frac{n}{q}}%
%TCIMACRO{\dprod \limits_{i=1}^{2}}%
%BeginExpansion
{\displaystyle \prod \limits_{i=1}^{2}}
%EndExpansion
\Vert b_{i}\Vert_{\ast}\left \Vert f_{i}\right \Vert _{L_{p_{i}}\left(
2B\right)  }\int \limits_{2r}^{\infty}\frac{dt}{t^{\frac{n}{q}+1}}\\
&  \lesssim%
%TCIMACRO{\dprod \limits_{i=1}^{2}}%
%BeginExpansion
{\displaystyle \prod \limits_{i=1}^{2}}
%EndExpansion
\Vert b_{i}\Vert_{\ast}r^{\frac{n}{q}}\int \limits_{2r}^{\infty}\prod
\limits_{i=1}^{2}\left \Vert f_{i}\right \Vert _{L_{p_{i}}\left(  B\left(
x_{0},t\right)  \right)  }\frac{dt}{t^{\frac{n}{q}+1}}\\
&  \lesssim%
%TCIMACRO{\dprod \limits_{i=1}^{2}}%
%BeginExpansion
{\displaystyle \prod \limits_{i=1}^{2}}
%EndExpansion
\Vert b_{i}\Vert_{\ast}r^{\frac{n}{q}}\int \limits_{2r}^{\infty}\left(
1+\ln \frac{t}{r}\right)  ^{2}%
%TCIMACRO{\dprod \limits_{i=1}^{2}}%
%BeginExpansion
{\displaystyle \prod \limits_{i=1}^{2}}
%EndExpansion
\Vert f_{i}\Vert_{L_{p_{i}}(B\left(  x_{0},t\right)  )}\frac{dt}{t^{n\left(
\frac{1}{q}-\left(  \frac{1}{q_{1}}+\frac{1}{q_{2}}\right)  \right)  +1}}.
\end{align*}

Secondly, for $F_{2}=\left \Vert I_{\alpha,\left(  b_{1},b_{2}\right)
}^{\left(  2\right)  }\left(  f_{1}^{0},f_{2}^{\infty}\right)  \right \Vert
_{L_{q}\left(  B\left(  x_{0},r\right)  \right)  }$, we decompose it into four
parts as follows:%
\begin{align*}
F_{2}  &  \lesssim \left \Vert \left[  \left(  b_{1}-\left \{  b_{1}\right \}
_{B}\right)  \right]  \left[  \left(  b_{2}-\left \{  b_{2}\right \}
_{B}\right)  \right]  I_{\alpha}^{\left(  2\right)  }\left(  f_{1}^{0}%
,f_{2}^{\infty}\right)  \right \Vert _{L_{q}\left(  B\left(  x_{0},r\right)
\right)  }\\
&  +\left \Vert \left[  \left(  b_{1}-\left \{  b_{1}\right \}  _{B}\right)
\right]  I_{\alpha}^{\left(  2\right)  }\left[  f_{1}^{0},\left(
b_{2}-\left \{  b_{2}\right \}  _{B}\right)  f_{2}^{\infty}\right]  \right \Vert
_{L_{q}\left(  B\left(  x_{0},r\right)  \right)  }\\
&  +\left \Vert \left[  \left(  b_{2}-\left \{  b_{2}\right \}  _{B}\right)
\right]  I_{\alpha}^{\left(  2\right)  }\left[  \left(  b_{1}-\left \{
b_{1}\right \}  _{B}\right)  f_{1}^{0},f_{2}^{\infty}\right]  \right \Vert
_{L_{q}\left(  B\left(  x_{0},r\right)  \right)  }\\
&  +\left \Vert I_{\alpha}^{\left(  2\right)  }\left[  \left(  b_{1}-\left \{
b_{1}\right \}  _{B}\right)  f_{1}^{0},\left(  b_{2}-\left \{  b_{2}\right \}
_{B}\right)  f_{2}^{\infty}\right]  \right \Vert _{L_{q}\left(  B\left(
x_{0},r\right)  \right)  }\\
&  \equiv F_{21}+F_{22}+F_{23}+F_{24}.
\end{align*}
Let $1<p_{1},p_{2}<\frac{2n}{\alpha}$, such that $\frac{1}{\overline{p}}%
=\frac{1}{p_{1}}+\frac{1}{p_{2}}$, $\frac{1}{\overline{q}}=\frac{1}%
{\overline{p}}-\frac{\alpha}{n}$, $\frac{1}{\overline{r}}=\frac{1}{q_{1}%
}+\frac{1}{q_{2}}$ and $\frac{1}{q}=\frac{1}{\overline{r}}+\frac{1}%
{\overline{q}}$. Then, using H\"{o}lder's inequality and by (\ref{1*}) we have%
\begin{align*}
F_{21}  &  =\left \Vert \left(  b_{1}-\left(  b_{1}\right)  _{B}\right)
\left(  b_{2}-\left(  b_{2}\right)  _{B}\right)  I_{\alpha}^{\left(  2\right)
}\left(  f_{1}^{0},f_{2}^{\infty}\right)  \right \Vert _{L_{q}\left(  B\left(
x_{0},r\right)  \right)  }\\
&  \lesssim \left \Vert \left(  b_{1}-\left(  b_{1}\right)  _{B}\right)  \left(
b_{2}-\left(  b_{2}\right)  _{B}\right)  \right \Vert _{L_{\overline{r}}\left(
B\left(  x_{0},r\right)  \right)  }\left \Vert I_{\alpha}^{\left(  2\right)
}\left(  f_{1}^{0},f_{2}^{\infty}\right)  \right \Vert _{L_{\overline{q}%
}\left(  B\left(  x_{0},r\right)  \right)  }\\
&  \lesssim \left \Vert b_{1}-\left(  b_{1}\right)  _{B}\right \Vert _{L_{q_{1}%
}\left(  B\left(  x_{0},r\right)  \right)  }\left \Vert b_{2}-\left(
b_{2}\right)  _{B}\right \Vert _{L_{q_{2}}\left(  B\left(  x_{0},r\right)
\right)  }\\
&  \times r^{\frac{n}{\overline{q}}}\int \limits_{2r}^{\infty}\left(
1+\ln \frac{t}{r}\right)  ^{2}\prod \limits_{i=1}^{2}\left \Vert f_{i}\right \Vert
_{L_{p_{i}}\left(  B\left(  x_{0},t\right)  \right)  }\frac{dt}{t^{\frac
{n}{\overline{q}}+1}}\\
&  \lesssim%
%TCIMACRO{\dprod \limits_{i=1}^{2}}%
%BeginExpansion
{\displaystyle \prod \limits_{i=1}^{2}}
%EndExpansion
\Vert b_{i}\Vert_{\ast}|B(x_{0},r)|^{\frac{1}{q_{1}}+\frac{1}{q_{2}}}%
r^{\frac{n}{\overline{q}}}\int \limits_{2r}^{\infty}\left(  1+\ln \frac{t}%
{r}\right)  ^{2}\prod \limits_{i=1}^{2}\left \Vert f_{i}\right \Vert _{L_{p_{i}%
}\left(  B\left(  x_{0},t\right)  \right)  }\frac{dt}{t^{\frac{n}{\overline
{q}}+1}}\\
&  \lesssim%
%TCIMACRO{\dprod \limits_{i=1}^{2}}%
%BeginExpansion
{\displaystyle \prod \limits_{i=1}^{2}}
%EndExpansion
\Vert b_{i}\Vert_{\ast}r^{n\left(  \frac{1}{q_{1}}+\frac{1}{q_{2}}\right)
}r^{n\left(  \frac{1}{p_{1}}+\frac{1}{p_{2}}-\frac{\alpha}{n}\right)  }%
\int \limits_{2r}^{\infty}\left(  1+\ln \frac{t}{r}\right)  ^{2}\prod
\limits_{i=1}^{2}\left \Vert f_{i}\right \Vert _{L_{p_{i}}\left(  B\left(
x_{0},t\right)  \right)  }t^{n\left(  \frac{1}{q_{1}}+\frac{1}{q_{2}}\right)
}\frac{dt}{t^{\frac{n}{q}+1}}\\
&  =%
%TCIMACRO{\dprod \limits_{i=1}^{2}}%
%BeginExpansion
{\displaystyle \prod \limits_{i=1}^{2}}
%EndExpansion
\Vert b_{i}\Vert_{\ast}r^{\frac{n}{q}}\int \limits_{2r}^{\infty}\left(
1+\ln \frac{t}{r}\right)  ^{2}%
%TCIMACRO{\dprod \limits_{i=1}^{2}}%
%BeginExpansion
{\displaystyle \prod \limits_{i=1}^{2}}
%EndExpansion
\Vert f_{i}\Vert_{L_{p_{i}}(B\left(  x_{0},t\right)  )}\frac{dt}{t^{n\left(
\frac{1}{q}-\left(  \frac{1}{q_{1}}+\frac{1}{q_{2}}\right)  \right)  +1}},
\end{align*}
where in the second inequality we have used the following fact:

It is clear that $\left \vert \left(  x_{0}-y_{1},\text{ }x_{0}-y_{2}\right)
\right \vert ^{2n-\alpha}\geq \left \vert x_{0}-y_{2}\right \vert ^{2n-\alpha}$.
By H\"{o}lder's inequality, we have%
\begin{align*}
\left \vert I_{\alpha}^{\left(  2\right)  }\left(  f_{1}^{0},f_{2}^{\infty
}\right)  \left(  x\right)  \right \vert  &  \lesssim%
%TCIMACRO{\dint \limits_{{\mathbb{R}^{n}}}}%
%BeginExpansion
{\displaystyle \int \limits_{{\mathbb{R}^{n}}}}
%EndExpansion%
%TCIMACRO{\dint \limits_{{\mathbb{R}^{n}}}}%
%BeginExpansion
{\displaystyle \int \limits_{{\mathbb{R}^{n}}}}
%EndExpansion
\frac{\left \vert f_{1}^{0}\left(  y_{1}\right)  \right \vert \left \vert
f_{2}^{\infty}\left(  y_{2}\right)  \right \vert }{\left \vert \left(
x-y_{1},x-y_{2}\right)  \right \vert ^{2n-\alpha}}dy_{1}dy_{2}\\
&  \lesssim \int \limits_{2B}\left \vert f_{1}\left(  y_{1}\right)  \right \vert
dy_{1}\int \limits_{\left(  2B\right)  ^{c}}\frac{\left \vert f_{2}\left(
y_{2}\right)  \right \vert }{\left \vert x_{0}-y_{2}\right \vert ^{2n-\alpha}%
}dy_{2}\\
&  \approx \int \limits_{2B}\left \vert f_{1}\left(  y_{1}\right)  \right \vert
dy_{1}\int \limits_{\left(  2B\right)  ^{c}}\left \vert f_{2}\left(
y_{2}\right)  \right \vert \int \limits_{\left \vert x_{0}-y_{2}\right \vert
}^{\infty}\frac{dt}{t^{2n-\alpha+1}}dy_{2}\\
&  \lesssim \left \Vert f_{1}\right \Vert _{L_{p_{1}}\left(  2B\right)
}\left \vert 2B\right \vert ^{1-\frac{1}{p_{1}}}%
%TCIMACRO{\dint \limits_{2r}^{\infty}}%
%BeginExpansion
{\displaystyle \int \limits_{2r}^{\infty}}
%EndExpansion
\left \Vert f_{2}\right \Vert _{L_{p_{2}}\left(  B\left(  x_{0},t\right)
\right)  }\left \vert B\left(  x_{0},t\right)  \right \vert ^{1-\frac{1}{p_{2}}%
}\frac{dt}{t^{2n-\alpha+1}}\\
&  \lesssim \int \limits_{2r}^{\infty}\prod \limits_{i=1}^{2}\left \Vert
f_{i}\right \Vert _{L_{p_{i}}\left(  B\left(  x_{0},t\right)  \right)  }%
\frac{dt}{t^{\frac{n}{\overline{q}}+1}},
\end{align*}
where $\frac{1}{p}=\frac{1}{p_{1}}+\frac{1}{p_{2}}$. Thus, the inequality
\[
\left \Vert I_{\alpha}^{\left(  2\right)  }\left(  f_{1}^{0},f_{2}^{\infty
}\right)  \right \Vert _{L_{\overline{q}}\left(  B\left(  x_{0},r\right)
\right)  }\lesssim r^{\frac{n}{\overline{q}}}\int \limits_{2r}^{\infty}%
\prod \limits_{i=1}^{2}\left \Vert f_{i}\right \Vert _{L_{p_{i}}\left(  B\left(
x_{0},t\right)  \right)  }\frac{dt}{t^{\frac{n}{\overline{q}}+1}}%
\]
is valid.

On the other hand, for the estimates used in $F_{22}$, $F_{23}$, we have to
prove the below inequality:%
\begin{equation}
\left \vert I_{\alpha}^{\left(  2\right)  }\left[  f_{1}^{0},\left(
b_{2}\left(  \cdot \right)  -\left(  b_{2}\right)  _{B}\right)  f_{2}^{\infty
}\right]  \left(  x\right)  \right \vert \lesssim \Vert b_{2}\Vert_{\ast}%
\int \limits_{2r}^{\infty}\left(  1+\ln \frac{t}{r}\right)  ^{2}\prod
\limits_{i=1}^{2}\left \Vert f_{i}\right \Vert _{L_{p_{i}}\left(  B\left(
x_{0},t\right)  \right)  }\frac{dt}{t^{n\left(  \frac{1}{p_{1}}+\frac{1}%
{p_{2}}\right)  +1-\alpha}}. \label{F}%
\end{equation}
To estimate $F_{22}$, the following inequality%
\[
\left \vert I_{\alpha}^{\left(  2\right)  }\left[  f_{1}^{0},\left(
b_{2}\left(  \cdot \right)  -\left(  b_{2}\right)  _{B}\right)  f_{2}^{\infty
}\right]  \left(  x\right)  \right \vert \lesssim%
%TCIMACRO{\dint \limits_{2B}}%
%BeginExpansion
{\displaystyle \int \limits_{2B}}
%EndExpansion
\left \vert f_{1}\left(  y_{1}\right)  \right \vert dy_{1}%
%TCIMACRO{\dint \limits_{\left(  2B\right)  ^{c}}}%
%BeginExpansion
{\displaystyle \int \limits_{\left(  2B\right)  ^{c}}}
%EndExpansion
\frac{\left \vert b_{2}\left(  y_{2}\right)  -\left(  b_{2}\right)
_{B}\right \vert \left \vert f_{2}\left(  y_{2}\right)  \right \vert }{\left \vert
x_{0}-y_{2}\right \vert ^{2n-\alpha}}dy_{2}%
\]
is satisfied. It's obvious that%
\begin{equation}%
%TCIMACRO{\dint \limits_{2B}}%
%BeginExpansion
{\displaystyle \int \limits_{2B}}
%EndExpansion
\left \vert f_{1}\left(  y_{1}\right)  \right \vert dy_{1}\lesssim \left \Vert
f_{1}\right \Vert _{L_{p_{1}}\left(  2B\right)  }\left \vert 2B\right \vert
^{1-\frac{1}{p_{1}}}, \label{48}%
\end{equation}
and using H\"{o}lder's inequality and by (\ref{5.2}) and (\ref{c}) we have%
\begin{align}
&
%TCIMACRO{\dint \limits_{\left(  2B\right)  ^{c}}}%
%BeginExpansion
{\displaystyle \int \limits_{\left(  2B\right)  ^{c}}}
%EndExpansion
\frac{\left \vert b_{2}\left(  y_{2}\right)  -\left(  b_{2}\right)
_{B}\right \vert \left \vert f_{2}\left(  y_{2}\right)  \right \vert }{\left \vert
x_{0}-y_{2}\right \vert ^{2n-\alpha}}dy_{2}\nonumber \\
&  \lesssim%
%TCIMACRO{\dint \limits_{\left(  2B\right)  ^{c}}}%
%BeginExpansion
{\displaystyle \int \limits_{\left(  2B\right)  ^{c}}}
%EndExpansion
\left \vert b_{2}\left(  y_{2}\right)  -\left(  b_{2}\right)  _{B\left(
x_{0},r\right)  }\right \vert \left \vert f_{2}\left(  y_{2}\right)  \right \vert
\left[
%TCIMACRO{\dint \limits_{\left\vert x_{0}-y_{2}\right\vert }^{\infty}}%
%BeginExpansion
{\displaystyle \int \limits_{\left \vert x_{0}-y_{2}\right \vert }^{\infty}}
%EndExpansion
\frac{dt}{t^{2n-\alpha+1}}\right]  dy_{2}\nonumber \\
&  \lesssim%
%TCIMACRO{\dint \limits_{2r}^{\infty}}%
%BeginExpansion
{\displaystyle \int \limits_{2r}^{\infty}}
%EndExpansion
\left \Vert b_{2}\left(  y_{2}\right)  -\left(  b_{2}\right)  _{B\left(
x_{0},t\right)  }\right \Vert _{L_{q_{2}}\left(  B\left(  x_{0},t\right)
\right)  }\left \Vert f_{2}\right \Vert _{L_{p_{2}}\left(  B\left(
x_{0},t\right)  \right)  }\left \vert B\left(  x_{0},t\right)  \right \vert
^{1-\left(  \frac{1}{p_{2}}+\frac{1}{q_{2}}\right)  }\frac{dt}{t^{2n-\alpha
+1}}\nonumber \\
&  +%
%TCIMACRO{\dint \limits_{2r}^{\infty}}%
%BeginExpansion
{\displaystyle \int \limits_{2r}^{\infty}}
%EndExpansion
\left \vert \left(  b_{2}\right)  _{B\left(  x_{0},t\right)  }-\left(
b_{2}\right)  _{B\left(  x_{0},r\right)  }\right \vert \left \Vert
f_{2}\right \Vert _{L_{p_{2}}\left(  B\left(  x_{0},t\right)  \right)
}\left \vert B\left(  x_{0},t\right)  \right \vert ^{1-\frac{1}{p_{2}}}\frac
{dt}{t^{2n-\alpha+1}}\nonumber \\
&  \lesssim \Vert b_{2}\Vert_{\ast}%
%TCIMACRO{\dint \limits_{2r}^{\infty}}%
%BeginExpansion
{\displaystyle \int \limits_{2r}^{\infty}}
%EndExpansion
\left(  1+\ln \frac{t}{r}\right)  ^{2}\left \vert B\left(  x_{0},t\right)
\right \vert ^{\frac{1}{q_{2}}}\left \Vert f_{2}\right \Vert _{L_{p_{2}}\left(
B\left(  x_{0},t\right)  \right)  }\left \vert B\left(  x_{0},t\right)
\right \vert ^{1-\left(  \frac{1}{p_{2}}+\frac{1}{q_{2}}\right)  }\frac
{dt}{t^{2n-\alpha+1}}\nonumber \\
&  +\Vert b_{2}\Vert_{\ast}\int \limits_{2r}^{\infty}\left(  1+\ln \frac{t}%
{r}\right)  \left \vert B\left(  x_{0},t\right)  \right \vert \left \Vert
f_{2}\right \Vert _{L_{p_{2}}\left(  B\left(  x_{0},t\right)  \right)
}\left \vert B\left(  x_{0},t\right)  \right \vert ^{1-\frac{1}{p_{2}}}\frac
{dt}{t^{2n-\alpha+1}}\nonumber \\
&  \lesssim \Vert b_{2}\Vert_{\ast}\int \limits_{2r}^{\infty}\left(  1+\ln
\frac{t}{r}\right)  ^{2}\left \Vert f_{2}\right \Vert _{L_{p_{2}}\left(
B\left(  x_{0},t\right)  \right)  }\frac{dt}{t^{n\left(  1+\frac{1}{p_{2}%
}\right)  +1-\alpha}}. \label{49}%
\end{align}
Hence, by (\ref{48}) and (\ref{49}), it follows that:%
\begin{align*}
&  \left \vert I_{\alpha}^{\left(  2\right)  }\left[  f_{1}^{0},\left(
b_{2}\left(  \cdot \right)  -\left(  b_{2}\right)  _{B}\right)  f_{2}^{\infty
}\right]  \left(  x\right)  \right \vert \\
&  \lesssim \Vert b_{2}\Vert_{\ast}\left \Vert f_{1}\right \Vert _{L_{p_{1}%
}\left(  2B\right)  }\left \vert 2B\right \vert ^{1-\frac{1}{p_{1}}}%
\int \limits_{2r}^{\infty}\left(  1+\ln \frac{t}{r}\right)  ^{2}\left \Vert
f_{2}\right \Vert _{L_{p_{2}}\left(  B\left(  x_{0},t\right)  \right)  }%
\frac{dt}{t^{n\left(  1+\frac{1}{p_{2}}\right)  +1-\alpha}}\\
&  \lesssim \Vert b_{2}\Vert_{\ast}\int \limits_{2r}^{\infty}\left(  1+\ln
\frac{t}{r}\right)  ^{2}\prod \limits_{i=1}^{2}\left \Vert f_{i}\right \Vert
_{L_{p_{i}}\left(  B\left(  x_{0},t\right)  \right)  }\frac{dt}{t^{n\left(
\frac{1}{p_{1}}+\frac{1}{p_{2}}\right)  +1-\alpha}}.
\end{align*}
This completes the proof of inequality (\ref{F}). Thus, let $1<\tau<\infty$,
such that $\frac{1}{q}=\frac{1}{q_{1}}+\frac{1}{\tau}$. Then, using
H\"{o}lder's inequality and from (\ref{F}) and (\ref{c}), we get%
\begin{align*}
F_{22}  &  =\left \Vert \left[  \left(  b_{1}-\left \{  b_{1}\right \}
_{B}\right)  \right]  I_{\alpha}^{\left(  2\right)  }\left[  f_{1}^{0},\left(
b_{2}-\left \{  b_{2}\right \}  _{B}\right)  f_{2}^{\infty}\right]  \right \Vert
_{L_{q}\left(  B\left(  x_{0},r\right)  \right)  }\\
&  \lesssim \left \Vert b_{1}-\left(  b_{1}\right)  _{B}\right \Vert _{L_{q_{1}%
}\left(  B\right)  }\left \Vert I_{\alpha}^{\left(  2\right)  }\left[
f_{1}^{0},\left(  b_{2}-\left(  b_{2}\right)  _{B}\right)  f_{2}^{\infty
}\right]  \right \Vert _{L_{\tau}\left(  B\right)  }\\
&  \lesssim%
%TCIMACRO{\dprod \limits_{i=1}^{2}}%
%BeginExpansion
{\displaystyle \prod \limits_{i=1}^{2}}
%EndExpansion
\Vert b_{i}\Vert_{\ast}\left \vert B\left(  x_{0},r\right)  \right \vert
^{\frac{1}{q_{1}}+\frac{1}{\tau}}\\
&  \times \int \limits_{2r}^{\infty}\left(  1+\ln \frac{t}{r}\right)  ^{2}%
\prod \limits_{i=1}^{2}\left \Vert f_{i}\right \Vert _{L_{p_{i}}\left(  B\left(
x_{0},t\right)  \right)  }\frac{dt}{t^{n\left(  \frac{1}{p_{1}}+\frac{1}%
{p_{2}}\right)  +1-\alpha}}\\
&  \lesssim%
%TCIMACRO{\dprod \limits_{i=1}^{2}}%
%BeginExpansion
{\displaystyle \prod \limits_{i=1}^{2}}
%EndExpansion
\Vert b_{i}\Vert_{\ast}r^{\frac{n}{q}}\int \limits_{2r}^{\infty}\left(
1+\ln \frac{t}{r}\right)  ^{2}%
%TCIMACRO{\dprod \limits_{i=1}^{2}}%
%BeginExpansion
{\displaystyle \prod \limits_{i=1}^{2}}
%EndExpansion
\Vert f_{i}\Vert_{L_{p_{i}}(B\left(  x_{0},t\right)  )}\frac{dt}{t^{n\left(
\frac{1}{q}-\left(  \frac{1}{q_{1}}+\frac{1}{q_{2}}\right)  \right)  +1}}.
\end{align*}

Similarly, $F_{23}$ has the same estimate above, here we omit the details,
thus the inequality%
\begin{align*}
F_{23}  &  =\left \Vert \left[  \left(  b_{2}-\left \{  b_{2}\right \}
_{B}\right)  \right]  I_{\alpha}^{\left(  2\right)  }\left[  \left(
b_{1}-\left \{  b_{1}\right \}  _{B}\right)  f_{1}^{0},f_{2}^{\infty}\right]
\right \Vert _{L_{q}\left(  B\left(  x_{0},r\right)  \right)  }\\
&  \lesssim%
%TCIMACRO{\dprod \limits_{i=1}^{2}}%
%BeginExpansion
{\displaystyle \prod \limits_{i=1}^{2}}
%EndExpansion
\Vert b_{i}\Vert_{\ast}r^{\frac{n}{q}}\int \limits_{2r}^{\infty}\left(
1+\ln \frac{t}{r}\right)  ^{2}%
%TCIMACRO{\dprod \limits_{i=1}^{2}}%
%BeginExpansion
{\displaystyle \prod \limits_{i=1}^{2}}
%EndExpansion
\Vert f_{i}\Vert_{L_{p_{i}}(B\left(  x_{0},t\right)  )}\frac{dt}{t^{n\left(
\frac{1}{q}-\left(  \frac{1}{q_{1}}+\frac{1}{q_{2}}\right)  \right)  +1}}%
\end{align*}
is valid.

Now we turn to estimate $F_{24}$. Similar to (\ref{F}), we have to prove the
following estimate for $F_{24}$:%
\begin{equation}
\left \vert I_{\alpha}^{\left(  2\right)  }\left[  \left(  b_{1}-\left(
b_{1}\right)  _{B}\right)  f_{1}^{0},\left(  b_{2}-\left(  b_{2}\right)
_{B}\right)  f_{2}^{\infty}\right]  \left(  x\right)  \right \vert \leq%
%TCIMACRO{\dprod \limits_{i=1}^{2}}%
%BeginExpansion
{\displaystyle \prod \limits_{i=1}^{2}}
%EndExpansion
\Vert b_{i}\Vert_{\ast}\int \limits_{2r}^{\infty}\left(  1+\ln \frac{t}%
{r}\right)  ^{2}%
%TCIMACRO{\dprod \limits_{i=1}^{2}}%
%BeginExpansion
{\displaystyle \prod \limits_{i=1}^{2}}
%EndExpansion
\Vert f_{i}\Vert_{L_{p_{i}}(B\left(  x_{0},t\right)  )}\frac{dt}{t^{n\left(
\frac{1}{q}-\left(  \frac{1}{q_{1}}+\frac{1}{q_{2}}\right)  \right)  +1}}.
\label{G}%
\end{equation}

Firstly, the following inequality%
\[
\left \vert I_{\alpha}^{\left(  2\right)  }\left[  \left(  b_{1}-\left(
b_{1}\right)  _{B}\right)  f_{1}^{0},\left(  b_{2}-\left(  b_{2}\right)
_{B}\right)  f_{2}^{\infty}\right]  \left(  x\right)  \right \vert \lesssim%
%TCIMACRO{\dint \limits_{2B}}%
%BeginExpansion
{\displaystyle \int \limits_{2B}}
%EndExpansion
\left \vert b_{1}\left(  y_{1}\right)  -\left(  b_{1}\right)  _{B}\right \vert
\left \vert f_{1}\left(  y_{1}\right)  \right \vert dy_{1}%
%TCIMACRO{\dint \limits_{\left(  2B\right)  ^{c}}}%
%BeginExpansion
{\displaystyle \int \limits_{\left(  2B\right)  ^{c}}}
%EndExpansion
\frac{\left \vert b_{2}\left(  y_{2}\right)  -\left(  b_{2}\right)
_{B}\right \vert \left \vert f_{2}\left(  y_{2}\right)  \right \vert }{\left \vert
x_{0}-y_{2}\right \vert ^{2n-\alpha}}dy_{2}%
\]
is valid.

It's obvious that from H\"{o}lder's inequality and (\ref{c})
\begin{equation}%
%TCIMACRO{\dint \limits_{2B}}%
%BeginExpansion
{\displaystyle \int \limits_{2B}}
%EndExpansion
\left \vert b_{1}\left(  y_{1}\right)  -\left(  b_{1}\right)  _{B}\right \vert
\left \vert f_{1}\left(  y_{1}\right)  \right \vert dy_{1}\lesssim \Vert
b_{1}\Vert_{\ast}\left \vert B\left(  x_{0},r\right)  \right \vert ^{1-\frac
{1}{p_{1}}}\left \Vert f_{1}\right \Vert _{L_{p_{1}}\left(  2B\right)  }.
\label{410}%
\end{equation}
Then, by (\ref{49}) and (\ref{410}) we have%
\[
\left \vert I_{\alpha}^{\left(  2\right)  }\left[  \left(  b_{1}-\left(
b_{1}\right)  _{B}\right)  f_{1}^{0},\left(  b_{2}-\left(  b_{2}\right)
_{B}\right)  f_{2}^{\infty}\right]  \left(  x\right)  \right \vert \leq%
%TCIMACRO{\dprod \limits_{i=1}^{2}}%
%BeginExpansion
{\displaystyle \prod \limits_{i=1}^{2}}
%EndExpansion
\Vert b_{i}\Vert_{\ast}\int \limits_{2r}^{\infty}\left(  1+\ln \frac{t}%
{r}\right)  ^{2}%
%TCIMACRO{\dprod \limits_{i=1}^{2}}%
%BeginExpansion
{\displaystyle \prod \limits_{i=1}^{2}}
%EndExpansion
\Vert f_{i}\Vert_{L_{p_{i}}(B\left(  x_{0},t\right)  )}\frac{dt}{t^{n\left(
\frac{1}{q}-\left(  \frac{1}{q_{1}}+\frac{1}{q_{2}}\right)  \right)  +1}}.
\]
This completes the proof of inequality (\ref{G}). Therefore, by (\ref{G}) we
deduce that%
\begin{align*}
F_{24}  &  =\left \Vert I_{\alpha}^{\left(  2\right)  }\left[  \left(
b_{1}-\left(  b_{1}\right)  _{B}\right)  f_{1}^{0},\left(  b_{2}-\left(
b_{2}\right)  _{B}\right)  f_{2}^{\infty}\right]  \right \Vert _{L_{q}\left(
B\right)  }\\
&  \lesssim%
%TCIMACRO{\dprod \limits_{i=1}^{2}}%
%BeginExpansion
{\displaystyle \prod \limits_{i=1}^{2}}
%EndExpansion
\Vert b_{i}\Vert_{\ast}r^{\frac{n}{q}}\int \limits_{2r}^{\infty}\left(
1+\ln \frac{t}{r}\right)  ^{2}%
%TCIMACRO{\dprod \limits_{i=1}^{2}}%
%BeginExpansion
{\displaystyle \prod \limits_{i=1}^{2}}
%EndExpansion
\Vert f_{i}\Vert_{L_{p_{i}}(B\left(  x_{0},t\right)  )}\frac{dt}{t^{n\left(
\frac{1}{q}-\left(  \frac{1}{q_{1}}+\frac{1}{q_{2}}\right)  \right)  +1}}.
\end{align*}

Considering estimates $F_{21},$ $F_{22}$, $F_{23}$, $F_{24}$ together, we get
the desired conclusion%
\[
F_{2}=\left \Vert I_{\alpha,\left(  b_{1},b_{2}\right)  }^{\left(  2\right)
}\left(  f_{1}^{0},f_{2}^{\infty}\right)  \right \Vert _{L_{q}\left(  B\left(
x_{0},r\right)  \right)  }\lesssim%
%TCIMACRO{\dprod \limits_{i=1}^{2}}%
%BeginExpansion
{\displaystyle \prod \limits_{i=1}^{2}}
%EndExpansion
\Vert b_{i}\Vert_{\ast}r^{\frac{n}{q}}\int \limits_{2r}^{\infty}\left(
1+\ln \frac{t}{r}\right)  ^{2}%
%TCIMACRO{\dprod \limits_{i=1}^{2}}%
%BeginExpansion
{\displaystyle \prod \limits_{i=1}^{2}}
%EndExpansion
\Vert f_{i}\Vert_{L_{p_{i}}(B\left(  x_{0},t\right)  )}\frac{dt}{t^{n\left(
\frac{1}{q}-\left(  \frac{1}{q_{1}}+\frac{1}{q_{2}}\right)  \right)  +1}}.
\]

Similar to $F_{2}$, we can also get the estimates for $F_{3}$,
\[
F_{3}=\left \Vert I_{\alpha,\left(  b_{1},b_{2}\right)  }^{\left(  2\right)
}\left(  f_{1}^{\infty},f_{2}^{0}\right)  \right \Vert _{L_{q}\left(  B\left(
x_{0},r\right)  \right)  }\lesssim%
%TCIMACRO{\dprod \limits_{i=1}^{2}}%
%BeginExpansion
{\displaystyle \prod \limits_{i=1}^{2}}
%EndExpansion
\Vert b_{i}\Vert_{\ast}r^{\frac{n}{q}}\int \limits_{2r}^{\infty}\left(
1+\ln \frac{t}{r}\right)  ^{2}%
%TCIMACRO{\dprod \limits_{i=1}^{2}}%
%BeginExpansion
{\displaystyle \prod \limits_{i=1}^{2}}
%EndExpansion
\Vert f_{i}\Vert_{L_{p_{i}}(B\left(  x_{0},t\right)  )}\frac{dt}{t^{n\left(
\frac{1}{q}-\left(  \frac{1}{q_{1}}+\frac{1}{q_{2}}\right)  \right)  +1}}.
\]

At last, we consider the last term $F_{4}=\left \Vert I_{\alpha,\left(
b_{1},b_{2}\right)  }^{\left(  2\right)  }\left(  f_{1}^{\infty},f_{2}%
^{\infty}\right)  \right \Vert _{L_{q}\left(  B\left(  x_{0},r\right)  \right)
}$. We split $F_{4}$ in the following way:%
\[
F_{4}\lesssim F_{41}+F_{42}+F_{43}+F_{44},
\]
where%
\begin{align*}
F_{41}  &  =\left \Vert \left(  b_{1}-\left(  b_{1}\right)  _{B}\right)
\left(  b_{2}-\left(  b_{2}\right)  _{B}\right)  I_{\alpha}^{\left(  2\right)
}\left(  f_{1}^{\infty},f_{2}^{\infty}\right)  \right \Vert _{L_{q}\left(
B\right)  },\\
F_{42}  &  =\left \Vert \left(  b_{1}-\left(  b_{1}\right)  _{B}\right)
I_{\alpha}^{\left(  2\right)  }\left[  f_{1}^{\infty},\left(  b_{2}-\left(
b_{2}\right)  _{B}\right)  f_{2}^{\infty}\right]  \right \Vert _{L_{q}\left(
B\right)  },\\
F_{43}  &  =\left \Vert \left(  b_{2}-\left(  b_{2}\right)  _{B}\right)
I_{\alpha}^{\left(  2\right)  }\left[  \left(  b_{1}-\left(  b_{1}\right)
_{B}\right)  f_{1}^{\infty},f_{2}^{\infty}\right]  \right \Vert _{L_{q}\left(
B\right)  },\\
F_{44}  &  =\left \Vert I_{\alpha}^{\left(  2\right)  }\left[  \left(
b_{1}-\left(  b_{1}\right)  _{B}\right)  f_{1}^{\infty},\left(  b_{2}-\left(
b_{2}\right)  _{B}\right)  f_{2}^{\infty}\right]  \right \Vert _{L_{q}\left(
B\right)  }.
\end{align*}
Now, let us estimate $F_{41}$, $F_{42}$, $F_{43}$, $F_{44}$ respectively.

For the term $F_{41}$, let $1<\tau<\infty$, such that $\frac{1}{q}=\left(
\frac{1}{q_{1}}+\frac{1}{q_{2}}\right)  +\frac{1}{\tau}$, $\frac{1}{\tau
}=\frac{1}{p_{1}}+\frac{1}{p_{2}}-\frac{\alpha}{n}$. Then, by H\"{o}lder's
inequality and (\ref{1*}), we get%
\begin{align*}
F_{41}  &  =\left \Vert \left(  b_{1}-\left(  b_{1}\right)  _{B}\right)
\left(  b_{2}-\left(  b_{2}\right)  _{B}\right)  I_{\alpha}^{\left(  2\right)
}\left(  f_{1}^{\infty},f_{2}^{\infty}\right)  \right \Vert _{L_{q}\left(
B\right)  }\\
&  \lesssim \left \Vert b_{1}-\left(  b_{1}\right)  _{B}\right \Vert _{L_{q_{1}%
}\left(  B\right)  }\left \Vert b_{2}-\left(  b_{2}\right)  _{B}\right \Vert
_{L_{q_{2}}\left(  B\right)  }\left \Vert I_{\alpha}^{\left(  2\right)
}\left(  f_{1}^{\infty},f_{2}^{\infty}\right)  \right \Vert _{L_{\tau}\left(
B\right)  }\\
&  \lesssim%
%TCIMACRO{\dprod \limits_{i=1}^{2}}%
%BeginExpansion
{\displaystyle \prod \limits_{i=1}^{2}}
%EndExpansion
\Vert b_{i}\Vert_{\ast}\left \vert B\left(  x_{0},r\right)  \right \vert
^{\frac{1}{q_{1}}+\frac{1}{q_{2}}}r^{\frac{n}{\tau}}\int \limits_{2r}^{\infty
}\left(  1+\ln \frac{t}{r}\right)  ^{2}%
%TCIMACRO{\dprod \limits_{i=1}^{2}}%
%BeginExpansion
{\displaystyle \prod \limits_{i=1}^{2}}
%EndExpansion
\Vert f_{i}\Vert_{L_{p_{i}}(B\left(  x_{0},t\right)  )}\frac{dt}{t^{\frac
{n}{\tau}+1}}\\
&  \lesssim%
%TCIMACRO{\dprod \limits_{i=1}^{2}}%
%BeginExpansion
{\displaystyle \prod \limits_{i=1}^{2}}
%EndExpansion
\Vert b_{i}\Vert_{\ast}r^{\frac{n}{q}}\int \limits_{2r}^{\infty}\left(
1+\ln \frac{t}{r}\right)  ^{2}%
%TCIMACRO{\dprod \limits_{i=1}^{2}}%
%BeginExpansion
{\displaystyle \prod \limits_{i=1}^{2}}
%EndExpansion
\Vert f_{i}\Vert_{L_{p_{i}}(B\left(  x_{0},t\right)  )}\frac{dt}{t^{n\left(
\frac{1}{q}-\left(  \frac{1}{q_{1}}+\frac{1}{q_{2}}\right)  \right)  +1}},
\end{align*}
where in the second inequality we have used the following fact:

Noting that $\left \vert \left(  x_{0}-y_{1},\text{ }x_{0}-y_{2}\right)
\right \vert ^{2n-\alpha}\geq \left \vert x_{0}-y_{1}\right \vert ^{n-\frac
{\alpha}{2}}\left \vert x_{0}-y_{2}\right \vert ^{n-\frac{\alpha}{2}}$ and by
H\"{o}lder's inequality, we get%
\begin{align*}
&  \left \vert I_{\alpha}^{\left(  2\right)  }\left(  f_{1}^{\infty}%
,f_{2}^{\infty}\right)  \left(  x\right)  \right \vert \\
&  \lesssim%
%TCIMACRO{\dint \limits_{\mathbb{R}^{n}}}%
%BeginExpansion
{\displaystyle \int \limits_{\mathbb{R}^{n}}}
%EndExpansion%
%TCIMACRO{\dint \limits_{\mathbb{R}^{n}}}%
%BeginExpansion
{\displaystyle \int \limits_{\mathbb{R}^{n}}}
%EndExpansion
\frac{\left \vert f_{1}\left(  y_{1}\right)  \chi_{\left(  2B\right)  ^{c}%
}\right \vert \left \vert f_{2}\left(  y_{2}\right)  \chi_{\left(  B\right)
^{c}}\right \vert }{\left \vert \left(  x_{0}-y_{1},x_{0}-y_{2}\right)
\right \vert ^{2n-\alpha}}dy_{1}dy_{2}\\
&  \lesssim%
%TCIMACRO{\dint \limits_{\left(  2B\right)  ^{c}}}%
%BeginExpansion
{\displaystyle \int \limits_{\left(  2B\right)  ^{c}}}
%EndExpansion%
%TCIMACRO{\dint \limits_{\left(  2B\right)  ^{c}}}%
%BeginExpansion
{\displaystyle \int \limits_{\left(  2B\right)  ^{c}}}
%EndExpansion
\frac{\left \vert f_{1}\left(  y_{1}\right)  \right \vert \left \vert
f_{2}\left(  y_{2}\right)  \right \vert }{\left \vert x_{0}-y_{1}\right \vert
^{n-\frac{\alpha}{2}}\left \vert x_{0}-y_{2}\right \vert ^{n-\frac{\alpha}{2}}%
}dy_{1}dy_{2}\\
&  \lesssim%
%TCIMACRO{\dsum \limits_{j=1}^{\infty}}%
%BeginExpansion
{\displaystyle \sum \limits_{j=1}^{\infty}}
%EndExpansion%
%TCIMACRO{\dprod \limits_{i=1}^{2}}%
%BeginExpansion
{\displaystyle \prod \limits_{i=1}^{2}}
%EndExpansion%
%TCIMACRO{\dint \limits_{2^{j+1}B\backslash2^{j}B}}%
%BeginExpansion
{\displaystyle \int \limits_{2^{j+1}B\backslash2^{j}B}}
%EndExpansion
\frac{\left \vert f_{i}\left(  y_{i}\right)  \right \vert }{\left \vert
x_{0}-y_{i}\right \vert ^{n-\frac{\alpha}{2}}}dy_{i}\\
&  \lesssim%
%TCIMACRO{\dsum \limits_{j=1}^{\infty}}%
%BeginExpansion
{\displaystyle \sum \limits_{j=1}^{\infty}}
%EndExpansion%
%TCIMACRO{\dprod \limits_{i=1}^{2}}%
%BeginExpansion
{\displaystyle \prod \limits_{i=1}^{2}}
%EndExpansion
\left(  2^{j}r\right)  ^{-n+\frac{\alpha}{2}}%
%TCIMACRO{\dint \limits_{2^{j+1}B}}%
%BeginExpansion
{\displaystyle \int \limits_{2^{j+1}B}}
%EndExpansion
\left \vert f_{i}\left(  y_{i}\right)  \right \vert dy_{i}\\
&  \lesssim%
%TCIMACRO{\dsum \limits_{j=1}^{\infty}}%
%BeginExpansion
{\displaystyle \sum \limits_{j=1}^{\infty}}
%EndExpansion
\left(  2^{j}r\right)  ^{-2n+\alpha}%
%TCIMACRO{\dprod \limits_{i=1}^{2}}%
%BeginExpansion
{\displaystyle \prod \limits_{i=1}^{2}}
%EndExpansion
\left \Vert f_{i}\right \Vert _{L_{p_{i}}(2^{j+1}B)}\left \vert 2^{j+1}%
B\right \vert ^{1-\frac{1}{p_{i}}}\\
&  \lesssim%
%TCIMACRO{\dsum \limits_{j=1}^{\infty}}%
%BeginExpansion
{\displaystyle \sum \limits_{j=1}^{\infty}}
%EndExpansion%
%TCIMACRO{\dint \limits_{2^{j+1}r}^{2^{j+2}r}}%
%BeginExpansion
{\displaystyle \int \limits_{2^{j+1}r}^{2^{j+2}r}}
%EndExpansion
\left(  2^{j+1}r\right)  ^{-2n+\alpha-1}%
%TCIMACRO{\dprod \limits_{i=1}^{2}}%
%BeginExpansion
{\displaystyle \prod \limits_{i=1}^{2}}
%EndExpansion
\left \Vert f_{i}\right \Vert _{L_{p_{i}}(2^{j+1}B)}\left \vert 2^{j+1}%
B\right \vert ^{1-\frac{1}{p_{i}}}dt\\
&  \lesssim%
%TCIMACRO{\dsum \limits_{j=1}^{\infty}}%
%BeginExpansion
{\displaystyle \sum \limits_{j=1}^{\infty}}
%EndExpansion%
%TCIMACRO{\dint \limits_{2^{j+1}r}^{2^{j+2}r}}%
%BeginExpansion
{\displaystyle \int \limits_{2^{j+1}r}^{2^{j+2}r}}
%EndExpansion%
%TCIMACRO{\dprod \limits_{i=1}^{2}}%
%BeginExpansion
{\displaystyle \prod \limits_{i=1}^{2}}
%EndExpansion
\left \Vert f_{i}\right \Vert _{L_{p_{i}}(B\left(  x_{0},t\right)  )}\left \vert
B\left(  x_{0},t\right)  \right \vert ^{1-\frac{1}{p_{i}}}\frac{dt}%
{t^{2n+1-\alpha}}\\
&  \lesssim%
%TCIMACRO{\dint \limits_{2r}^{\infty}}%
%BeginExpansion
{\displaystyle \int \limits_{2r}^{\infty}}
%EndExpansion
\prod \limits_{i=1}^{2}\left \Vert f_{i}\right \Vert _{L_{p_{i}}\left(  B\left(
x_{0},t\right)  \right)  }\left \vert B\left(  x_{0},t\right)  \right \vert
^{2-\left(  \frac{1}{p_{1}}+\frac{1}{p_{2}}\right)  }\frac{dt}{t^{2n+1-\alpha
}}\\
&  \lesssim%
%TCIMACRO{\dint \limits_{2r}^{\infty}}%
%BeginExpansion
{\displaystyle \int \limits_{2r}^{\infty}}
%EndExpansion
\prod \limits_{i=1}^{2}\left \Vert f_{i}\right \Vert _{L_{p_{i}}\left(  B\left(
x_{0},t\right)  \right)  }\frac{dt}{t^{\frac{n}{\tau}+1}}.
\end{align*}
Moreover, for $p_{1}$, $p_{2}\in \left[  1,\infty \right)  $ the inequality
\begin{equation}
\left \Vert I_{\alpha}^{\left(  2\right)  }\left(  f_{1}^{\infty},f_{2}%
^{\infty}\right)  \right \Vert _{L_{\tau}\left(  B\left(  x_{0},r\right)
\right)  }\lesssim r^{\frac{n}{\tau}}\int \limits_{2r}^{\infty}\prod
\limits_{i=1}^{2}\left \Vert f_{i}\right \Vert _{L_{p_{i}}\left(  B\left(
x_{0},t\right)  \right)  }\frac{dt}{t^{\frac{n}{\tau}+1}} \label{0}%
\end{equation}
is valid.

For the terms $F_{42}$, $F_{43}$, similar to the estimates used for (\ref{F}),
we have to prove the following inequality:%
\begin{equation}
\left \vert I_{\alpha}^{\left(  2\right)  }\left[  f_{1}^{\infty},\left(
b_{2}-\left(  b_{2}\right)  _{B}\right)  f_{2}^{\infty}\right]  \left(
x\right)  \right \vert \lesssim \Vert b_{2}\Vert_{\ast}\int \limits_{2r}^{\infty
}\left(  1+\ln \frac{t}{r}\right)  ^{2}\prod \limits_{i=1}^{2}\left \Vert
f_{i}\right \Vert _{L_{p_{i}}\left(  B\left(  x_{0},t\right)  \right)  }%
\frac{dt}{t^{n\left(  \frac{1}{p_{1}}+\frac{1}{p_{2}}\right)  +1-\alpha}}.
\label{11}%
\end{equation}

Noting that $\left \vert \left(  x_{0}-y_{1},\text{ }x_{0}-y_{2}\right)
\right \vert ^{2n-\alpha}\geq \left \vert x_{0}-y_{1}\right \vert ^{n-\frac
{\alpha}{2}}\left \vert x_{0}-y_{2}\right \vert ^{n-\frac{\alpha}{2}}$, we get%
\begin{align*}
&  \left \vert I_{\alpha}^{\left(  2\right)  }\left[  f_{1}^{\infty},\left(
b_{2}-\left(  b_{2}\right)  _{B}\right)  f_{2}^{\infty}\right]  \left(
x\right)  \right \vert \\
&  \lesssim%
%TCIMACRO{\dint \limits_{\mathbb{R}^{n}}}%
%BeginExpansion
{\displaystyle \int \limits_{\mathbb{R}^{n}}}
%EndExpansion%
%TCIMACRO{\dint \limits_{\mathbb{R}^{n}}}%
%BeginExpansion
{\displaystyle \int \limits_{\mathbb{R}^{n}}}
%EndExpansion
\frac{\left \vert b_{2}\left(  y_{2}\right)  -\left(  b_{2}\right)
_{B}\right \vert \left \vert f_{1}\left(  y_{1}\right)  \chi_{\left(  2B\right)
^{c}}\right \vert \left \vert f_{2}\left(  y_{2}\right)  \chi_{\left(
2B\right)  ^{c}}\right \vert }{\left \vert \left(  x_{0}-y_{1},x_{0}%
-y_{2}\right)  \right \vert ^{2n-\alpha}}dy_{1}dy_{2}\\
&  \lesssim%
%TCIMACRO{\dint \limits_{\left(  2B\right)  ^{c}}}%
%BeginExpansion
{\displaystyle \int \limits_{\left(  2B\right)  ^{c}}}
%EndExpansion%
%TCIMACRO{\dint \limits_{\left(  2B\right)  ^{c}}}%
%BeginExpansion
{\displaystyle \int \limits_{\left(  2B\right)  ^{c}}}
%EndExpansion
\frac{\left \vert b_{2}\left(  y_{2}\right)  -\left(  b_{2}\right)
_{B}\right \vert \left \vert f_{1}\left(  y_{1}\right)  \right \vert \left \vert
f_{2}\left(  y_{2}\right)  \right \vert }{\left \vert x_{0}-y_{1}\right \vert
^{n-\frac{\alpha}{2}}\left \vert x_{0}-y_{2}\right \vert ^{n-\frac{\alpha}{2}}%
}dy_{1}dy_{2}\\
&  \lesssim%
%TCIMACRO{\dsum \limits_{j=1}^{\infty}}%
%BeginExpansion
{\displaystyle \sum \limits_{j=1}^{\infty}}
%EndExpansion%
%TCIMACRO{\dint \limits_{2^{j+1}B\backslash2^{j}B}}%
%BeginExpansion
{\displaystyle \int \limits_{2^{j+1}B\backslash2^{j}B}}
%EndExpansion
\frac{\left \vert f_{1}\left(  y_{1}\right)  \right \vert }{\left \vert
x_{0}-y_{1}\right \vert ^{n-\frac{\alpha}{2}}}dy_{1}%
%TCIMACRO{\dint \limits_{2^{j+1}B\backslash2^{j}B}}%
%BeginExpansion
{\displaystyle \int \limits_{2^{j+1}B\backslash2^{j}B}}
%EndExpansion
\frac{\left \vert b_{2}\left(  y_{2}\right)  -\left(  b_{2}\right)
_{B}\right \vert \left \vert f_{2}\left(  y_{2}\right)  \right \vert }{\left \vert
x_{0}-y_{2}\right \vert ^{n-\frac{\alpha}{2}}}dy_{2}\\
&  \lesssim%
%TCIMACRO{\dsum \limits_{j=1}^{\infty}}%
%BeginExpansion
{\displaystyle \sum \limits_{j=1}^{\infty}}
%EndExpansion
\left(  2^{j}r\right)  ^{-2n+\alpha}%
%TCIMACRO{\dint \limits_{2^{j+1}B}}%
%BeginExpansion
{\displaystyle \int \limits_{2^{j+1}B}}
%EndExpansion
\left \vert f_{1}\left(  y_{1}\right)  \right \vert dy_{1}%
%TCIMACRO{\dint \limits_{2^{j+1}B}}%
%BeginExpansion
{\displaystyle \int \limits_{2^{j+1}B}}
%EndExpansion
\left \vert b_{2}\left(  y_{2}\right)  -\left(  b_{2}\right)  _{B}\right \vert
\left \vert f_{2}\left(  y_{2}\right)  \right \vert dy_{2}.
\end{align*}
On the other hand, it's obvious that%
\begin{equation}%
%TCIMACRO{\dint \limits_{2^{j+1}B}}%
%BeginExpansion
{\displaystyle \int \limits_{2^{j+1}B}}
%EndExpansion
\left \vert f_{1}\left(  y_{1}\right)  \right \vert dy_{1}\leq \left \Vert
f_{1}\right \Vert _{L_{p_{1}}(2^{j+1}B)}\left \vert 2^{j+1}B\right \vert
^{1-\frac{1}{p_{1}}}, \label{f1}%
\end{equation}
and using H\"{o}lder's inequality and by (\ref{5.2}) and (\ref{c})
\begin{align}
&
%TCIMACRO{\dint \limits_{2^{j+1}B}}%
%BeginExpansion
{\displaystyle \int \limits_{2^{j+1}B}}
%EndExpansion
\left \vert b_{2}\left(  y_{2}\right)  -\left(  b_{2}\right)  _{B}\right \vert
\left \vert f_{2}\left(  y_{2}\right)  \right \vert dy_{2}\nonumber \\
&  \leq \left \Vert b_{2}-\left(  b_{2}\right)  _{2^{j+1}B}\right \Vert
_{L_{q_{2}}(2^{j+1}B)}\left \Vert f_{2}\right \Vert _{L_{p_{2}}(2^{j+1}%
B)}\left \vert 2^{j+1}B\right \vert ^{1-\left(  \frac{1}{p_{2}}+\frac{1}{q_{2}%
}\right)  }\nonumber \\
&  +\left \vert \left(  b_{2}\right)  _{2^{j+1}B}-\left(  b_{2}\right)
_{B}\right \vert \left \Vert f_{2}\right \Vert _{L_{p_{2}}(2^{j+1}B)}\left \vert
2^{j+1}B\right \vert ^{1-\frac{1}{p_{2}}}\nonumber \\
&  \lesssim \Vert b_{2}\Vert_{\ast}\left \vert 2^{j+1}B\right \vert ^{\frac
{1}{q_{2}}}\left(  1+\ln \frac{2^{j+1}r}{r}\right)  \left \Vert f_{2}\right \Vert
_{L_{p_{2}}\left(  2^{j+1}B\right)  }\left \vert 2^{j+1}B\right \vert
^{1-\left(  \frac{1}{p_{2}}+\frac{1}{q_{2}}\right)  }\nonumber \\
&  +\Vert b_{2}\Vert_{\ast}\left(  1+\ln \frac{2^{j+1}r}{r}\right)  \left \vert
2^{j+1}B\right \vert \left \Vert f_{2}\right \Vert _{L_{p_{2}}\left(
2^{j+1}B\right)  }\left \vert 2^{j+1}B\right \vert ^{1-\frac{1}{p_{2}}%
}\nonumber \\
&  \lesssim \Vert b_{2}\Vert_{\ast}\left(  1+\ln \frac{2^{j+1}r}{r}\right)
^{2}\left \vert 2^{j+1}B\right \vert ^{1-\frac{1}{p_{2}}}\left \Vert
f_{2}\right \Vert _{L_{p_{2}}\left(  2^{j+1}B\right)  }. \label{f2}%
\end{align}
Hence, by (\ref{f1}) and (\ref{f2}), it follows that:%
\begin{align*}
&  \left \vert I_{\alpha}^{\left(  2\right)  }\left[  f_{1}^{\infty},\left(
b_{2}-\left(  b_{2}\right)  _{B}\right)  f_{2}^{\infty}\right]  \left(
x\right)  \right \vert \\
&  \lesssim%
%TCIMACRO{\dsum \limits_{j=1}^{\infty}}%
%BeginExpansion
{\displaystyle \sum \limits_{j=1}^{\infty}}
%EndExpansion
\left(  2^{j}r\right)  ^{-2n+\alpha}%
%TCIMACRO{\dint \limits_{2^{j+1}B}}%
%BeginExpansion
{\displaystyle \int \limits_{2^{j+1}B}}
%EndExpansion
\left \vert f_{1}\left(  y_{1}\right)  \right \vert dy_{1}%
%TCIMACRO{\dint \limits_{2^{j+1}B}}%
%BeginExpansion
{\displaystyle \int \limits_{2^{j+1}B}}
%EndExpansion
\left \vert b_{2}\left(  y_{2}\right)  -\left(  b_{2}\right)  _{B}\right \vert
\left \vert f_{2}\left(  y_{2}\right)  \right \vert dy_{2}\\
&  \lesssim \Vert b_{2}\Vert_{\ast}%
%TCIMACRO{\dsum \limits_{j=1}^{\infty}}%
%BeginExpansion
{\displaystyle \sum \limits_{j=1}^{\infty}}
%EndExpansion
\left(  2^{j}r\right)  ^{-2n+\alpha}\left(  1+\ln \frac{2^{j+1}r}{r}\right)
^{2}\left \vert 2^{j+1}B\right \vert ^{2-\left(  \frac{1}{p_{1}}+\frac{1}{p_{2}%
}\right)  }\prod \limits_{i=1}^{2}\left \Vert f_{i}\right \Vert _{L_{p_{i}%
}\left(  2^{j+1}B\right)  }\\
&  \lesssim \Vert b_{2}\Vert_{\ast}%
%TCIMACRO{\dsum \limits_{j=1}^{\infty}}%
%BeginExpansion
{\displaystyle \sum \limits_{j=1}^{\infty}}
%EndExpansion
\int \limits_{2^{j+1}r}^{2^{j+2}r}\left(  2^{j+1}r\right)  ^{-2n+\alpha
-1}\left(  1+\ln \frac{2^{j+1}r}{r}\right)  ^{2}\left \vert 2^{j+1}B\right \vert
^{2-\left(  \frac{1}{p_{1}}+\frac{1}{p_{2}}\right)  }\prod \limits_{i=1}%
^{2}\left \Vert f_{i}\right \Vert _{L_{p_{i}}\left(  2^{j+1}B\right)  }dt\\
&  \lesssim \Vert b_{2}\Vert_{\ast}%
%TCIMACRO{\dsum \limits_{j=1}^{\infty}}%
%BeginExpansion
{\displaystyle \sum \limits_{j=1}^{\infty}}
%EndExpansion
\int \limits_{2^{j+1}r}^{2^{j+2}r}\left(  1+\ln \frac{2^{j+1}r}{r}\right)
^{2}\left \vert 2^{j+1}B\right \vert ^{2-\left(  \frac{1}{p_{1}}+\frac{1}{p_{2}%
}\right)  }\prod \limits_{i=1}^{2}\left \Vert f_{i}\right \Vert _{L_{p_{i}%
}\left(  2^{j+1}B\right)  }\frac{dt}{t^{2n-\alpha+1}}\\
&  \lesssim \Vert b_{2}\Vert_{\ast}\int \limits_{2r}^{\infty}\left(  1+\ln
\frac{t}{r}\right)  ^{2}\left \vert B\left(  x_{0},t\right)  \right \vert
^{2-\left(  \frac{1}{p_{1}}+\frac{1}{p_{2}}\right)  }\prod \limits_{i=1}%
^{2}\left \Vert f_{i}\right \Vert _{L_{p_{i}}\left(  B\left(  x_{0},t\right)
\right)  }\frac{dt}{t^{2n-\alpha+1}}\\
&  \lesssim \Vert b_{2}\Vert_{\ast}\int \limits_{2r}^{\infty}\left(  1+\ln
\frac{t}{r}\right)  ^{2}\prod \limits_{i=1}^{2}\left \Vert f_{i}\right \Vert
_{L_{p_{i}}\left(  B\left(  x_{0},t\right)  \right)  }\frac{dt}{t^{n\left(
\frac{1}{p_{1}}+\frac{1}{p_{2}}\right)  +1-\alpha}}.
\end{align*}
This completes the proof of (\ref{11}).

Now we turn to estimate $F_{42}$. Let $1<\tau<\infty$, such that $\frac{1}%
{q}=\frac{1}{q_{1}}+\frac{1}{\tau}$. Then, by H\"{o}lder's inequality,
(\ref{c}) and (\ref{11}), we obtain%
\begin{align*}
F_{42}  &  =\left \Vert \left(  b_{1}-\left(  b_{1}\right)  _{B}\right)
I_{\alpha}^{\left(  2\right)  }\left[  f_{1}^{\infty},\left(  b_{2}-\left(
b_{2}\right)  _{B}\right)  f_{2}^{\infty}\right]  \right \Vert _{L_{q}\left(
B\right)  }\\
&  \lesssim \left \Vert \left(  b_{1}-\left(  b_{1}\right)  _{B}\right)
\right \Vert _{L_{q_{1}}\left(  B\right)  }\left \Vert I_{\alpha}^{\left(
2\right)  }\left[  f_{1}^{\infty},\left(  b_{2}-\left(  b_{2}\right)
_{B}\right)  f_{2}^{\infty}\right]  \right \Vert _{L_{\tau \left(  B\right)  }%
}\\
&  \lesssim%
%TCIMACRO{\dprod \limits_{i=1}^{2}}%
%BeginExpansion
{\displaystyle \prod \limits_{i=1}^{2}}
%EndExpansion
\Vert b_{i}\Vert_{\ast}r^{\frac{n}{q}}\int \limits_{2r}^{\infty}\left(
1+\ln \frac{t}{r}\right)  ^{2}%
%TCIMACRO{\dprod \limits_{i=1}^{2}}%
%BeginExpansion
{\displaystyle \prod \limits_{i=1}^{2}}
%EndExpansion
\Vert f_{i}\Vert_{L_{p_{i}}(B\left(  x_{0},t\right)  )}\frac{dt}{t^{n\left(
\frac{1}{q}-\left(  \frac{1}{q_{1}}+\frac{1}{q_{2}}\right)  \right)  +1}}.
\end{align*}

Similarly, $F_{43}$ has the same estimate above, here we omit the details,
thus the inequality%
\begin{align*}
F_{43}  &  =\left \Vert \left(  b_{2}-\left(  b_{2}\right)  _{B}\right)
I_{\alpha}^{\left(  2\right)  }\left[  \left(  b_{1}-\left(  b_{1}\right)
_{B}\right)  f_{1}^{\infty},f_{2}^{\infty}\right]  \right \Vert _{L_{q}\left(
B\right)  }\\
&  \lesssim%
%TCIMACRO{\dprod \limits_{i=1}^{2}}%
%BeginExpansion
{\displaystyle \prod \limits_{i=1}^{2}}
%EndExpansion
\Vert b_{i}\Vert_{\ast}r^{\frac{n}{q}}\int \limits_{2r}^{\infty}\left(
1+\ln \frac{t}{r}\right)  ^{2}%
%TCIMACRO{\dprod \limits_{i=1}^{2}}%
%BeginExpansion
{\displaystyle \prod \limits_{i=1}^{2}}
%EndExpansion
\Vert f_{i}\Vert_{L_{p_{i}}(B\left(  x_{0},t\right)  )}\frac{dt}{t^{n\left(
\frac{1}{q}-\left(  \frac{1}{q_{1}}+\frac{1}{q_{2}}\right)  \right)  +1}}%
\end{align*}
is valid.

Finally, to estimate $F_{44}$, similar to the estimate of (\ref{11}), we have%
\begin{align*}
&  \left \vert I_{\alpha}^{\left(  2\right)  }\left[  \left(  b_{1}-\left(
b_{2}\right)  _{B}\right)  f_{1}^{\infty},\left(  b_{2}-\left(  b_{2}\right)
_{B}\right)  f_{2}^{\infty}\right]  \left(  x\right)  \right \vert \\
&  \lesssim%
%TCIMACRO{\dsum \limits_{j=1}^{\infty}}%
%BeginExpansion
{\displaystyle \sum \limits_{j=1}^{\infty}}
%EndExpansion
\left(  2^{j}r\right)  ^{-2n+\alpha}%
%TCIMACRO{\dint \limits_{2^{j+1}B}}%
%BeginExpansion
{\displaystyle \int \limits_{2^{j+1}B}}
%EndExpansion
\left \vert b_{1}\left(  y_{1}\right)  -\left(  b_{1}\right)  _{B}\right \vert
\left \vert f_{1}\left(  y_{1}\right)  \right \vert dy_{1}%
%TCIMACRO{\dint \limits_{2^{j+1}B}}%
%BeginExpansion
{\displaystyle \int \limits_{2^{j+1}B}}
%EndExpansion
\left \vert b_{2}\left(  y_{2}\right)  -\left(  b_{2}\right)  _{B}\right \vert
\left \vert f_{2}\left(  y_{2}\right)  \right \vert dy_{2}\\
&  \lesssim%
%TCIMACRO{\dprod \limits_{i=1}^{2}}%
%BeginExpansion
{\displaystyle \prod \limits_{i=1}^{2}}
%EndExpansion
\Vert b_{i}\Vert_{\ast}\int \limits_{2r}^{\infty}\left(  1+\ln \frac{t}%
{r}\right)  ^{2}%
%TCIMACRO{\dprod \limits_{i=1}^{2}}%
%BeginExpansion
{\displaystyle \prod \limits_{i=1}^{2}}
%EndExpansion
\Vert f_{i}\Vert_{L_{p_{i}}(B\left(  x_{0},t\right)  )}\frac{dt}{t^{n\left(
\frac{1}{q}-\left(  \frac{1}{q_{1}}+\frac{1}{q_{2}}\right)  \right)  +1}}.
\end{align*}

Thus, we have%
\begin{align*}
F_{44}  &  =\left \Vert I_{\alpha}^{\left(  2\right)  }\left[  \left(
b_{1}-\left(  b_{1}\right)  _{B}\right)  f_{1}^{\infty},\left(  b_{2}-\left(
b_{2}\right)  _{B}\right)  f_{2}^{\infty}\right]  \right \Vert _{L_{q}\left(
B\right)  }\\
&  \lesssim%
%TCIMACRO{\dprod \limits_{i=1}^{2}}%
%BeginExpansion
{\displaystyle \prod \limits_{i=1}^{2}}
%EndExpansion
\Vert b_{i}\Vert_{\ast}r^{\frac{n}{q}}\int \limits_{2r}^{\infty}\left(
1+\ln \frac{t}{r}\right)  ^{2}%
%TCIMACRO{\dprod \limits_{i=1}^{2}}%
%BeginExpansion
{\displaystyle \prod \limits_{i=1}^{2}}
%EndExpansion
\Vert f_{i}\Vert_{L_{p_{i}}(B\left(  x_{0},t\right)  )}\frac{dt}{t^{n\left(
\frac{1}{q}-\left(  \frac{1}{q_{1}}+\frac{1}{q_{2}}\right)  \right)  +1}}.
\end{align*}

By the estimates of $F_{4j}$ above, where $j=1$,$2$,$3$,$4$, we know that%
\[
F_{4}=\left \Vert I_{\alpha,\left(  b_{1},b_{2}\right)  }^{\left(  2\right)
}\left(  f_{1}^{\infty},f_{2}^{\infty}\right)  \right \Vert _{L_{q}\left(
B\left(  x_{0},r\right)  \right)  }\lesssim%
%TCIMACRO{\dprod \limits_{i=1}^{2}}%
%BeginExpansion
{\displaystyle \prod \limits_{i=1}^{2}}
%EndExpansion
\Vert b_{i}\Vert_{\ast}r^{\frac{n}{q}}\int \limits_{2r}^{\infty}\left(
1+\ln \frac{t}{r}\right)  ^{2}%
%TCIMACRO{\dprod \limits_{i=1}^{2}}%
%BeginExpansion
{\displaystyle \prod \limits_{i=1}^{2}}
%EndExpansion
\Vert f_{i}\Vert_{L_{p_{i}}(B\left(  x_{0},t\right)  )}\frac{dt}{t^{n\left(
\frac{1}{q}-\left(  \frac{1}{q_{1}}+\frac{1}{q_{2}}\right)  \right)  +1}}.
\]

Consequently, combining all the estimates for $F_{1},$ $F_{2}$, $F_{3}$,
$F_{4}$, we complete the proof of Lemma \ref{Lemma}.
\end{proof}

\section{Proofs of the main results}

Now we are ready to return to the proofs of Theorems \ref{teo1} and \ref{teo2}.

\subsection{\textbf{Proof of Theorem \ref{teo1}.}}

\begin{proof}
To prove Theorem \ref{teo1}, we will use the following relationship between
essential supremum and essential infimum%
\begin{equation}
\left(  \operatorname*{essinf}\limits_{x\in E}f\left(  x\right)  \right)
^{-1}=\operatorname*{esssup}\limits_{x\in E}\frac{1}{f\left(  x\right)  },
\label{5}%
\end{equation}
where $f$ is any real-valued nonnegative function and measurable on $E$ (see
\cite{Wheeden-Zygmund}, page 143). Indeed, we consider (\ref{51}) firstly.

Since $\overrightarrow{f}\in M_{p_{1},\varphi_{1}}\times \cdots \times
M_{p_{m},\varphi_{m}}$, by (\ref{5}) and the non-decreasing, with respect to
$t$, of the norm $%
%TCIMACRO{\dprod \limits_{i=1}^{m}}%
%BeginExpansion
{\displaystyle \prod \limits_{i=1}^{m}}
%EndExpansion
\Vert f_{i}\Vert_{L_{p_{i}}(B\left(  x,t\right)  )}$, we get%
\begin{align}
&  \frac{%
%TCIMACRO{\dprod \limits_{i=1}^{m}}%
%BeginExpansion
{\displaystyle \prod \limits_{i=1}^{m}}
%EndExpansion
\Vert f_{i}\Vert_{L_{p_{i}}(B\left(  x,t\right)  )}}{\operatorname*{essinf}%
\limits_{0<t<\tau<\infty}%
%TCIMACRO{\dprod \limits_{i=1}^{m}}%
%BeginExpansion
{\displaystyle \prod \limits_{i=1}^{m}}
%EndExpansion
\varphi_{i}(x,\tau)\tau^{\frac{n}{p}}}\leq \operatorname*{esssup}%
\limits_{0<t<\tau<\infty}\frac{%
%TCIMACRO{\dprod \limits_{i=1}^{m}}%
%BeginExpansion
{\displaystyle \prod \limits_{i=1}^{m}}
%EndExpansion
\Vert f_{i}\Vert_{L_{p_{i}}(B\left(  x,t\right)  )}}{%
%TCIMACRO{\dprod \limits_{i=1}^{m}}%
%BeginExpansion
{\displaystyle \prod \limits_{i=1}^{m}}
%EndExpansion
\varphi_{i}(x,\tau)\tau^{\frac{n}{p}}}\nonumber \\
&  \leq \operatorname*{esssup}\limits_{0<\tau<\infty,x\in{\mathbb{R}^{n}}}%
\frac{%
%TCIMACRO{\dprod \limits_{i=1}^{m}}%
%BeginExpansion
{\displaystyle \prod \limits_{i=1}^{m}}
%EndExpansion
\Vert f_{i}\Vert_{L_{p_{i}}(B\left(  x,t\right)  )}}{%
%TCIMACRO{\dprod \limits_{i=1}^{m}}%
%BeginExpansion
{\displaystyle \prod \limits_{i=1}^{m}}
%EndExpansion
\varphi_{i}(x,\tau)\tau^{\frac{n}{p}}}\leq%
%TCIMACRO{\dprod \limits_{i=1}^{m}}%
%BeginExpansion
{\displaystyle \prod \limits_{i=1}^{m}}
%EndExpansion
\left \Vert f_{i}\right \Vert _{M_{p_{i},\varphi_{i}}}. \label{9}%
\end{align}
For $1<p_{1},\ldots,p_{m}<\infty$, since $(\varphi_{1},\ldots,\varphi
_{m},\varphi)$ satisfies (\ref{50}) and by (\ref{9}), we have%
\begin{align}
&  \int \limits_{r}^{\infty}\left(  1+\ln \frac{t}{r}\right)  ^{m}%
%TCIMACRO{\dprod \limits_{i=1}^{m}}%
%BeginExpansion
{\displaystyle \prod \limits_{i=1}^{m}}
%EndExpansion
\Vert f_{i}\Vert_{L_{p_{i}}(B\left(  x,t\right)  )}\frac{dt}{t^{n\left(
\frac{1}{q}-%
%TCIMACRO{\dsum \limits_{i=1}^{m}}%
%BeginExpansion
{\displaystyle \sum \limits_{i=1}^{m}}
%EndExpansion
\frac{1}{q_{i}}\right)  +1}}\nonumber \\
&  \leq \int \limits_{r}^{\infty}\left(  1+\ln \frac{t}{r}\right)  ^{m}\frac{%
%TCIMACRO{\dprod \limits_{i=1}^{m}}%
%BeginExpansion
{\displaystyle \prod \limits_{i=1}^{m}}
%EndExpansion
\Vert f_{i}\Vert_{L_{p_{i}}(B\left(  x,t\right)  )}}{\operatorname*{essinf}%
\limits_{t<\tau<\infty}%
%TCIMACRO{\dprod \limits_{i=1}^{m}}%
%BeginExpansion
{\displaystyle \prod \limits_{i=1}^{m}}
%EndExpansion
\varphi_{i}(x,\tau)\tau^{\frac{n}{p}}}\frac{\operatorname*{essinf}%
\limits_{t<\tau<\infty}%
%TCIMACRO{\dprod \limits_{i=1}^{m}}%
%BeginExpansion
{\displaystyle \prod \limits_{i=1}^{m}}
%EndExpansion
\varphi_{i}(x,\tau)\tau^{\frac{n}{p}}}{t^{n\left(  \frac{1}{q}-%
%TCIMACRO{\dsum \limits_{i=1}^{m}}%
%BeginExpansion
{\displaystyle \sum \limits_{i=1}^{m}}
%EndExpansion
\frac{1}{q_{i}}\right)  +1}}dt\nonumber \\
&  \leq C%
%TCIMACRO{\dprod \limits_{i=1}^{m}}%
%BeginExpansion
{\displaystyle \prod \limits_{i=1}^{m}}
%EndExpansion
\left \Vert f_{i}\right \Vert _{M_{p_{i},\varphi_{i}}}\int \limits_{r}^{\infty
}\left(  1+\ln \frac{t}{r}\right)  ^{m}\frac{\operatorname*{essinf}%
\limits_{t<\tau<\infty}%
%TCIMACRO{\dprod \limits_{i=1}^{m}}%
%BeginExpansion
{\displaystyle \prod \limits_{i=1}^{m}}
%EndExpansion
\varphi_{i}(x,\tau)\tau^{\frac{n}{p}}}{t^{n\left(  \frac{1}{q}-%
%TCIMACRO{\dsum \limits_{i=1}^{m}}%
%BeginExpansion
{\displaystyle \sum \limits_{i=1}^{m}}
%EndExpansion
\frac{1}{q_{i}}\right)  +1}}dt\nonumber \\
&  \leq C%
%TCIMACRO{\dprod \limits_{i=1}^{m}}%
%BeginExpansion
{\displaystyle \prod \limits_{i=1}^{m}}
%EndExpansion
\left \Vert f_{i}\right \Vert _{M_{p_{i},\varphi_{i}}}\varphi(x,r). \label{13}%
\end{align}
Then by (\ref{200}) and (\ref{13}), we get%
\begin{align*}
\left \Vert I_{\alpha,\overrightarrow{b}}^{\left(  m\right)  }\left(
\overrightarrow{f}\right)  \right \Vert _{M_{q,\varphi}}  &  =\sup
\limits_{x\in{\mathbb{R}^{n}},r>0}\varphi \left(  x,r\right)  ^{-1}%
|B(x,r)|^{-\frac{1}{q}}\left \Vert I_{\alpha,\overrightarrow{b}}^{\left(
m\right)  }\left(  \overrightarrow{f}\right)  \right \Vert _{L_{q}\left(
B(x,r)\right)  }\\
&  \lesssim%
%TCIMACRO{\dprod \limits_{i=1}^{m}}%
%BeginExpansion
{\displaystyle \prod \limits_{i=1}^{m}}
%EndExpansion
\Vert b_{i}\Vert_{\ast}\sup_{x\in{\mathbb{R}^{n},}r>0}\varphi \left(
x_{0},r\right)  ^{-1}\int \limits_{r}^{\infty}\left(  1+\ln \frac{t}{r}\right)
^{m}%
%TCIMACRO{\dprod \limits_{i=1}^{m}}%
%BeginExpansion
{\displaystyle \prod \limits_{i=1}^{m}}
%EndExpansion
\Vert f_{i}\Vert_{L_{p_{i}}(B\left(  x_{0},t\right)  )}\frac{dt}{t^{n\left(
\frac{1}{q}-%
%TCIMACRO{\dsum \limits_{i=1}^{m}}%
%BeginExpansion
{\displaystyle \sum \limits_{i=1}^{m}}
%EndExpansion
\frac{1}{q_{i}}\right)  +1}}\\
&  \lesssim%
%TCIMACRO{\dprod \limits_{i=1}^{m}}%
%BeginExpansion
{\displaystyle \prod \limits_{i=1}^{m}}
%EndExpansion
\Vert b_{i}\Vert_{\ast}\left \Vert f_{i}\right \Vert _{M_{p_{i},\varphi_{i}}}.
\end{align*}
Thus we obtain (\ref{51}).

The conclusion of (\ref{52}) is a direct consequence of (\ref{4}) and
(\ref{51}). Indeed, from the process proving (\ref{51}), it is easy to see
that the conclusions of (\ref{51}) also hold for $\widetilde{I}%
_{\overrightarrow{b},\alpha}^{\left(  m\right)  }$. Combining this with
(\ref{4}), we can immediately obtain (\ref{52}), which completes the proof.
\end{proof}

\subsection{\textbf{Proof of Theorem \ref{teo2}.}}

\begin{proof}
Since the inequalities (\ref{53}) and (\ref{54}) hold by Theorem \ref{teo1},
we only have to prove the implication%
\begin{equation}
\lim \limits_{r\rightarrow0}\sup \limits_{x\in{\mathbb{R}^{n}}}\frac
{r^{-\frac{n}{p}}%
%TCIMACRO{\dprod \limits_{i=1}^{m}}%
%BeginExpansion
{\displaystyle \prod \limits_{i=1}^{m}}
%EndExpansion
\Vert f_{i}\Vert_{L_{p_{i}}(B\left(  x,r\right)  )}}{%
%TCIMACRO{\dprod \limits_{i=1}^{m}}%
%BeginExpansion
{\displaystyle \prod \limits_{i=1}^{m}}
%EndExpansion
\varphi_{i}(x,r)}=0\text{ implies }\lim \limits_{r\rightarrow0}\sup
\limits_{x\in{\mathbb{R}^{n}}}\frac{r^{-\frac{n}{q}}\left \Vert I_{\alpha
,\overrightarrow{b}}^{\left(  m\right)  }\left(  \overrightarrow{f}\right)
\right \Vert _{L_{q}\left(  B\left(  x,r\right)  \right)  }}{\varphi(x,r)}=0.
\label{16*}%
\end{equation}

To show that%
\[
\sup \limits_{x\in{\mathbb{R}^{n}}}\frac{r^{-\frac{n}{q}}\left \Vert
I_{\alpha,\overrightarrow{b}}^{\left(  m\right)  }\left(  \overrightarrow
{f}\right)  \right \Vert _{L_{q}\left(  B\left(  x,r\right)  \right)  }%
}{\varphi(x,r)}<\varepsilon \text{ for small }r,
\]
we use the estimate (\ref{200}):%
\[
\sup \limits_{x\in{\mathbb{R}^{n}}}\frac{r^{-\frac{n}{q}}\left \Vert
I_{\alpha,\overrightarrow{b}}^{\left(  m\right)  }\left(  \overrightarrow
{f}\right)  \right \Vert _{L_{q}\left(  B\left(  x,r\right)  \right)  }%
}{\varphi(x,r)}\lesssim \sup \limits_{x\in{\mathbb{R}^{n}}}\frac{%
%TCIMACRO{\dprod \limits_{i=1}^{m}}%
%BeginExpansion
{\displaystyle \prod \limits_{i=1}^{m}}
%EndExpansion
\Vert b_{i}\Vert_{\ast}}{\varphi(x,r)}\,%
%TCIMACRO{\dint \limits_{r}^{\infty}}%
%BeginExpansion
{\displaystyle \int \limits_{r}^{\infty}}
%EndExpansion
\left(  1+\ln \frac{t}{r}\right)  ^{m}%
%TCIMACRO{\dprod \limits_{i=1}^{m}}%
%BeginExpansion
{\displaystyle \prod \limits_{i=1}^{m}}
%EndExpansion
\Vert f_{i}\Vert_{L_{p_{i}}(B\left(  x,t\right)  )}\frac{dt}{t^{n\left(
\frac{1}{q}-%
%TCIMACRO{\dsum \limits_{i=1}^{m}}%
%BeginExpansion
{\displaystyle \sum \limits_{i=1}^{m}}
%EndExpansion
\frac{1}{q_{i}}\right)  +1}}.
\]

We take $r<\delta_{0}$, where $\delta_{0}$ is small enough and split the
integration:%
\begin{equation}
\frac{r^{-\frac{n}{q}}\left \Vert I_{\alpha,\overrightarrow{b}}^{\left(
m\right)  }\left(  \overrightarrow{f}\right)  \right \Vert _{L_{q}\left(
B\left(  x,r\right)  \right)  }}{\varphi(x,r)}\leq C\left[  I_{\delta_{0}%
}\left(  x,r\right)  +J_{\delta_{0}}\left(  x,r\right)  \right]  , \label{17*}%
\end{equation}
where $\delta_{0}>0$ (we may take $\delta_{0}<1$), and
\[
I_{\delta_{0}}\left(  x,r\right)  :=\frac{%
%TCIMACRO{\dprod \limits_{i=1}^{m}}%
%BeginExpansion
{\displaystyle \prod \limits_{i=1}^{m}}
%EndExpansion
\Vert b_{i}\Vert_{\ast}}{\varphi(x,r)}%
%TCIMACRO{\dint \limits_{r}^{\delta_{0}}}%
%BeginExpansion
{\displaystyle \int \limits_{r}^{\delta_{0}}}
%EndExpansion
\left(  1+\ln \frac{t}{r}\right)  ^{m}%
%TCIMACRO{\dprod \limits_{i=1}^{m}}%
%BeginExpansion
{\displaystyle \prod \limits_{i=1}^{m}}
%EndExpansion
\Vert f_{i}\Vert_{L_{p_{i}}(B\left(  x_{0},t\right)  )}\frac{dt}{t^{n\left(
\frac{1}{q}-%
%TCIMACRO{\dsum \limits_{i=1}^{m}}%
%BeginExpansion
{\displaystyle \sum \limits_{i=1}^{m}}
%EndExpansion
\frac{1}{q_{i}}\right)  +1}},
\]
and%
\[
J_{\delta_{0}}\left(  x,r\right)  :=\frac{%
%TCIMACRO{\dprod \limits_{i=1}^{m}}%
%BeginExpansion
{\displaystyle \prod \limits_{i=1}^{m}}
%EndExpansion
\Vert b_{i}\Vert_{\ast}}{\varphi(x,r)}%
%TCIMACRO{\dint \limits_{\delta_{0}}^{\infty}}%
%BeginExpansion
{\displaystyle \int \limits_{\delta_{0}}^{\infty}}
%EndExpansion
\left(  1+\ln \frac{t}{r}\right)  ^{m}%
%TCIMACRO{\dprod \limits_{i=1}^{m}}%
%BeginExpansion
{\displaystyle \prod \limits_{i=1}^{m}}
%EndExpansion
\Vert f_{i}\Vert_{L_{p_{i}}(B\left(  x_{0},t\right)  )}\frac{dt}{t^{n\left(
\frac{1}{q}-%
%TCIMACRO{\dsum \limits_{i=1}^{m}}%
%BeginExpansion
{\displaystyle \sum \limits_{i=1}^{m}}
%EndExpansion
\frac{1}{q_{i}}\right)  +1}}%
\]
and $r<\delta_{0}$. Now we can choose any fixed $\delta_{0}>0$ such that%
\[
\sup \limits_{x\in{\mathbb{R}^{n}}}\frac{t^{-\frac{n}{p}}%
%TCIMACRO{\dprod \limits_{i=1}^{m}}%
%BeginExpansion
{\displaystyle \prod \limits_{i=1}^{m}}
%EndExpansion
\Vert f_{i}\Vert_{L_{p_{i}}(B\left(  x,t\right)  )}}{%
%TCIMACRO{\dprod \limits_{i=1}^{m}}%
%BeginExpansion
{\displaystyle \prod \limits_{i=1}^{m}}
%EndExpansion
\varphi_{i}(x,t)}<\frac{\varepsilon}{2CC_{0}},\qquad t\leq \delta_{0},
\]
where $C$ and $C_{0}$ are constants from (\ref{10*}) and (\ref{17*}), which is
possible since $\overrightarrow{f}\in VM_{p_{1},\varphi_{1}}\times \cdots \times
VM_{p_{m},\varphi_{m}}$. This allows to estimate the first term uniformly in
$r\in \left(  0,\delta_{0}\right)  $:%
\[%
%TCIMACRO{\dprod \limits_{i=1}^{m}}%
%BeginExpansion
{\displaystyle \prod \limits_{i=1}^{m}}
%EndExpansion
\Vert b_{i}\Vert_{\ast}\sup \limits_{x\in{\mathbb{R}^{n}}}CI_{\delta_{0}%
}\left(  x,r\right)  <\frac{\varepsilon}{2},\qquad0<r<\delta_{0}%
\]
by (\ref{10*}).

For the second term, writing $1+\ln \frac{t}{r}\leq1+\left \vert \ln
t\right \vert +\ln \frac{1}{r}$, by the choice of $r$ sufficiently small because
of the conditions (\ref{11*}) we obtain%
\[
J_{\delta_{0}}\left(  x,r\right)  \leq \frac{c_{\delta_{0}}+\widetilde
{c_{\delta_{0}}}\ln \frac{1}{r}}{\varphi(x,r)}%
%TCIMACRO{\dprod \limits_{i=1}^{m}}%
%BeginExpansion
{\displaystyle \prod \limits_{i=1}^{m}}
%EndExpansion
\Vert b_{i}\Vert_{\ast}\left \Vert f_{i}\right \Vert _{VM_{p_{i},\varphi_{i}}},
\]
where $c_{\delta_{0}}$ is the constant from (\ref{12*}) with $\delta
=\delta_{0}$ and $\widetilde{c_{\delta_{0}}}$ is a similar constant with
omitted logarithmic factor in the integrand. Then, by (\ref{11*}) we can
choose small enough $r$ such that%
\[
\sup \limits_{x\in{\mathbb{R}^{n}}}J_{\delta_{0}}\left(  x,r\right)
<\frac{\varepsilon}{2},
\]
which completes the proof of (\ref{16*}).

For $M_{\alpha,\overrightarrow{b}}^{\left(  m\right)  }$, we can also use the
same method to obtain the desired result, but we omit the details. Therefore,
the proof of Theorem \ref{teo2} is completed.
\end{proof}


\begin{thebibliography}{99}                                                                                               %


\bibitem {Cao-Chen}X. N. Cao, D. X. Chen, \textit{The boundedness of
Toeplitz-type operators on vanishing Morrey spaces}, Anal. Theory Appl.,
\textbf{27} (2011), 309-319.

\bibitem {Chen1}D. X. Chen, J. Chen, S. Mao, \textit{Weighted }$L^{p}$\textit{
estimates for maximal commutators of multilinear singular integrals}, Chin.
Ann. Math., \textbf{34B} (6) (2013), 885-902.

\bibitem {Chen}X. Chen, Q. Y. Xue, \textit{Weighted estimates for a class of
multilinear fractional type operators}, J. Math. Anal. Appl., \textbf{362}
(2010), 355-373.

\bibitem {CM}R. R. Coifman, Y. Meyer, \textit{On commutators of singular
integrals and bilinear singular integrals}, Trans. Amer. Math. Soc.,
\textbf{212} (1975), 315-331.

\bibitem {Fu}Z. W. Fu, Y. Lin, S. Z. Lu, $\lambda$\textit{-Central }%
$BMO$\textit{\ estimates for commutators of singular integral operators with
rough kernel}, Acta Math. Sin. (Engl. Ser.), \textbf{24} (2008), 373-386.

\bibitem {Grafakos-sm}L. Grafakos,\textit{ On multilinear fractional
integrals}, Studia. Math., \textbf{102} (1992), 49-56.

\bibitem {Grafakos1}L. Grafakos, R. H. Torres, \textit{Multilinear
Calder\'{o}n-Zygmund theory}, Adv. Math., \textbf{165} (2002), 124-164.

\bibitem {Grafakos2}L. Grafakos, R. H. Torres, \textit{Maximal operator and
weighted norm inequalities for multilinear singular integrals}, Indiana Univ.
Math. J., \textbf{51} (2002), 1261-1276.

\bibitem {Grafakos3*}L. Grafakos, R. H. Torres, \textit{On multilinear
singular integrals of Calder\'{o}n-Zygmund type}, in: Proceedings of the 6th
International Conference on Harmonic Analysis and Partial Differential
Equations (El Escorial), in: Publ. Mat., vol. Extra, 2002, 57-91.

\bibitem {GulJIA}V. S. Guliyev, \textit{Boundedness of the maximal, potential
and singular operators in the generalized Morrey spaces}, J. Inequal. Appl.,
2009, Art. ID 503948, 20 pp.

\bibitem {Guliyev}V. S. Guliyev, S. S. Aliyev, T. Karaman, P. S. Shukurov,
\textit{Boundedness of sublinear operators and commutators on generalized
Morrey Space}, Integr. Equ. Oper. Theory, \textbf{71} (2011), 327-355.

\bibitem {Gurbuz1}F. Gurbuz, \textit{Boundedness of some potential type
sublinear operators and their commutators with rough kernels on generalized
local Morrey spaces }$\left[  \text{\textit{Ph.D. thesis}}\right]  $, Ankara
University, Ankara, Turkey, 2015 (in Turkish).

\bibitem {Gurbuz2}F. Gurbuz\textit{, Weighted Morrey and Weighted fractional
Sobolev-Morrey Spaces estimates for a large class of pseudo-differential
operators with smooth symbols}, J. Pseudo-Differ. Oper. Appl., \textbf{7} (4)
(2016), 595-607. doi:10.1007/s11868-016-0158-8.

\bibitem {Gurbuz3}F. Gurbuz\textit{, Sublinear operators with rough kernel
generated by Calder\'{o}n-Zygmund operators and their commutators on
generalized Morrey spaces}, Math. Notes, in press.

\bibitem {Gurbuz4}F. Gurbuz\textit{, Some estimates for generalized
commutators of rough fractional maximal and integral operators on generalized
weighted Morrey spaces}, Canad. Math. Bull., \textbf{60} (1) (2017), 131-145.

\bibitem {Gurbuz5}F. Gurbuz\textit{, Sublinear operators with a rough kernel
generated by fractional integrals and local Campanato space estimates for
commutators with rough kernel on generalized local Morrey spaces}, Int. J.
Appl. Math. \& Stat., \textbf{56} (3) (2017), 52-62.

\bibitem {John-Nirenberg}F. John and L. Nirenberg, \textit{On functions of
bounded mean oscillation}, Comm. Pure Appl. Math., \textbf{14} (1961), 415-426.

\bibitem {Karaman}T. Karaman, \textit{Boundedness of some classes of sublinear
operators on generalized weighted Morrey spaces and some applications
}$\left[  \text{\textit{Ph.D. thesis}}\right]  $, Ankara University, Ankara,
Turkey, 2012 (in Turkish).

\bibitem {Kenig}C. E. Kenig, E. M. Stein, \textit{Multilinear estimates and
fractional integration}, Math. Res. Lett., \textbf{6} (1999), 1-15.

\bibitem {Lin}Y. Lin,\  \textit{Strongly singular Calder\'{o}n-Zygmund operator
and commutator on Morrey type spaces}, Acta Math. Sin. (Engl. Ser.),
\textbf{23} (11) (2007), 2097-2110.

\bibitem {Miranda}C. Miranda,\textit{\ Sulle equazioni ellittiche del secondo
ordine di tipo non variazionale, a coefficienti discontinui.}, Ann. Math. Pura
E Appl., \textbf{63} (4) (1963), 353-386.

\bibitem {Miz}T. Mizuhara, \textit{Boundedness of some classical operators on
generalized Morrey spaces}, Harmonic Analysis (S. Igari, Editor), ICM 90
Satellite Proceedings, Springer - Verlag, Tokyo (1991), 183-189.

\bibitem {Moen}K. Moen, \textit{Weighted inequalities for multilinear
fractional integral operators}, Collect. Math., \textbf{60} (2009), 213-238.

\bibitem {Morrey}C. B. Morrey, \textit{On the solutions of quasi-linear
elliptic partial differential equations}, Trans. Amer. Math. Soc., \textbf{43}
(1938), 126-166.

\bibitem {Muckenhoupt}B. Muckenhoupt and R. L. Wheeden, \textit{Weighted norm
inequalities for fractional integrals}, Trans. Amer. Math. Soc., \textbf{192}
(1974), 261-274.

\bibitem {Pal}D. K. Palagachev, L. G. Softova, \textit{Singular integral
operators, Morrey spaces and fine regularity of solutions to PDE's}, Potential
Anal., \textbf{20} (2004), 237-263.

\bibitem {Si}Z. Y. Si, S. Z. Lu, \textit{Weighted estimates for iterated
commutators of multilinear fractional operators}, Acta Math. Sin. (Engl.
Ser.), \textbf{28} (9) (2012), 1769-1778.

\bibitem {Softova}L. G. Softova, \textit{Singular integrals and commutators in
generalized Morrey spaces}, Acta Math. Sin. (Engl. Ser.), \textbf{22} (3)
(2006), 757-766.

\bibitem {Vitanza1}C. Vitanza, \textit{Functions with vanishing Morrey norm
and elliptic partial differential equations}, in: Proceedings of Methods of
Real Analysis and Partial Differential Equations, Capri, pp. 147-150. Springer (1990).

\bibitem {Vitanza2}C. Vitanza, \textit{Regularity results for a class of
elliptic equations with coefficients in Morrey spaces}, Ricerche Mat.,
\textbf{42} (2) (1993), 265-281.

\bibitem {Xu}J. Xu, \textit{Boundedness in Lebesgue spaces for commutators of
multilinear singular integrals and }$RBMO$\textit{ functions with non-doubling
measures}, Sci. China (Series A), \textbf{50} (2007), 361-376.

\bibitem {Yu}X. Yu, X. X. Tao, \textit{Boundedness of multilinear operators on
generalized Morrey spaces}, Appl. Math. J. Chinese Univ., \textbf{29} (2)
(2014), 127-138.

\bibitem {Wheeden-Zygmund}R. L. Wheeden and A. Zygmund, \textit{Measure and
Integral: An Introduction to Real Analysis}, vol. 43 of Pure and Applied
Mathematics, Marcel Dekker, New York, NY, USA, 1977.
\end{thebibliography}
\end{document}